\documentclass[11pt]{amsart}
\usepackage{latexsym,amssymb,dsfont,textcomp}
\usepackage[centertags]{amsmath}
\usepackage{fullpage}
\usepackage{color}
\usepackage{epsfig,psfrag}

\usepackage[foot]{amsaddr}

\newtheorem{lemma}{Lemma}[section]
\newtheorem{theorem}[lemma]{Theorem}
\newtheorem{proposition}[lemma]{Proposition}

\newtheorem{remark}[lemma]{Remark}

\def\charf {\mbox{{\text 1}\kern-.24em {\text l}}}

\newcommand{\R}{\mathbb{R}}
\newcommand{\del}{\partial}

\newcommand{\bu}{\bar u}
\newcommand{\btheta}{\bar \theta}
\newcommand{\bs}{\bar \sigma}
\newcommand{\hu}{\hat u}
\newcommand{\htheta}{\hat \theta}
\newcommand{\hs}{\hat \sigma}
\newcommand{\fhu}{\hat u}
\newcommand{\fhtheta}{\hat \theta}

\newcommand{\hU}{\hat U}
\newcommand{\hTheta}{\hat \Theta}
\newcommand{\hS}{\hat \Sigma}
\newcommand{\bU}{\bar U}
\newcommand{\bTheta}{\bar \Theta}
\newcommand{\bS}{\bar \Sigma}
\newcommand{\tU}{\tilde U}
\newcommand{\tTheta}{\tilde \Theta}
\newcommand{\tS}{\tilde \Sigma}

\newcommand{\ii}{\mathrm{i}}

\newenvironment{system*}{\begin{equation*}\left\lbrace\begin{aligned}}{\end{aligned}\right.\end{equation*}}

\title{Emergence of coherent localized structures in shear deformations of temperature dependent fluids $^*$}
 \thanks{$^*$Research partially supported by the EU FP7-REGPOT project 
``Archimedes Center for
Modeling, Analysis and Computation" and the ``DIKICOMA"  project of the 
Hellenic Secretariat of Research and Technology.
Part of this work was completed at the Department of Applied Mathematics, University of Crete, Greece.}

\author{Theodoros Katsaounis$^{1,}$$^2$}
\author{Julier Olivier$^3$}
\author{Athanasios E. Tzavaras$^{1,}$$^2$}
\address{$^1$ Computer, Electrical, Mathematical Sciences \& Engineering Division, King Abdullah University of Science and Technology (KAUST), Thuwal, Saudi Arabia.
Email: {\tt athanasios.tzavaras@kaust.edu.sa} }
\address{$^2$ Institute of Applied and Computational Mathematics, FORTH, Heraklion 71110, Greece}
\address{$^3$ Centre de Math\'ematiques et Informatique, Aix-Marseille Universit\'e, 39 rue F. Joliot Curie, 13453 Marseille Cedex 13, France.
Email: {\tt julien.olivier@univ-amu.fr}}

\begin{document}

\date{}
\numberwithin{equation}{section}


\begin{abstract}
Shear localization occurs in various instances of material instability in solid mechanics and is typically
associated with  Hadamard-instability for an underlying model.
While Hadamard instability indicates the catastrophic growth of oscillations around a mean state, 
it does not by itself explain the formation of coherent structures typically observed in localization. The latter
is a nonlinear effect and its analysis is the main objective of this article. 
We consider a model that captures the main mechanisms observed
in  high strain-rate deformation of metals,
and describes shear motions of temperature dependent non-Newtonian fluids.
For a special dependence of the viscosity on the temperature, we carry out a linearized stability 
analysis around a base state of uniform shearing solutions, 
and quantitatively assess  the effects of the various mechanisms affecting the problem: 
thermal softening, momentum diffusion  and thermal diffusion. Then, we turn to the nonlinear model, 
and construct localized states - in the form of similarity solutions -  that emerge as coherent structures 
in the localization process. This justifies  a scenario
for localization that is proposed on the basis of asymptotic analysis in \cite{KT}.
\end{abstract}

\maketitle


\section{Introduction}
\label{intro}

The phenomenon of shear localization during high strain-rate deformations of metals \cite{ZH,CDHS,HDH}  
is a striking instance of material instability in mechanics, that has attracted considerable attention 
in the mechanics literature ({\it e.g.} \cite{CDHS,FM,MC,WW,Wr}) as well as
a number of mathematical studies   (\cite{DH,Tzavaras87, BPV,WrOc,Tzavaras92,DO,KT} and references therein).
Material instability is typically associated with ill-posedness  of an underlying initial value problem, 
what has coined the term  \emph{Hadamard instability} for its description 
in the mechanics literature. 
It should however be emphasized that while Hadamard instability indicates the catastrophic growth 
of oscillations around a mean state, what is observed in localization is the orderly albeit extremely fast
development of coherent structures: the shear bands. 
The latter is a nonlinear phenomenon and a linearized analysis can only serve as  a precursor
to such behavior.  

We  work with the simplest model capturing the mechanism of shear band formation,  a
simple shear motion of a non-Newtonian fluid with temperature decreasing viscosity $\mu' (\theta) < 0$,
\begin{equation}
  \label{introshear} 
  \begin{aligned}
    & v_{t} = \sigma_{x},\\
    & \theta_{t} = \kappa \theta_{ x x}  +  \sigma v_x, \\
    & \sigma  = \mu (\theta) v_x^n \, ,
  \end{aligned}
\end{equation}
where  $n > 0$ describes the strain-rate sensitivity  which  typically in 
these problems  satisfies $n \ll 1$.
Mathematical studies of \eqref{shear},  for power law  viscosities $\mu (\theta)  = \theta^{- \alpha}$, $\alpha > 0$, 
show that the uniform shear is  asymptotically stable for $n > \alpha$  \cite{DH, Tzavaras87},
but indicate unstable response in the complementary region $n < \alpha$,
\cite{Tzavaras87, BPV, Tzavaras92}.  The mathematical analysis has not so far characterized
the precise behavior  in the instability region; however, detailed and resolved numerical studies 
({\it e.g.} \cite{WW, ELW, BKT}) 
show that instability is followed by formation of shear bands.

The difficulty is that in order to analyze the unstable regime one needs to account for the combined
effect of parabolic regularizations in conjunction to ill-posed initial value problems.
An asymptotic theory for calculating an effective equation for the compound effect  is proposed in \cite{KT, KT2}.  
It is based on the theory of relaxation systems and the Chapman-Enskog expansion 
and produces an effective asymptotic equation. For a power law the effective equation
changes type from forward to backward parabolic, across the threshold $-\alpha + n = 0$ of the parameter space. 
The criteria of \cite{KT, KT2} are based on asymptotic analysis and as such are formal in nature.

Our objective, here, is to develop a mathematical theory for the onset of localisation and the emergence of a nonlinear
localized state, thus providing a rigorous justification to the behavior conjectured  by the effective equation in \cite{KT}. 
We employ the constitutive hypothesis of an  {\it exponential law} for the viscosity coefficient, 
\begin{equation}
\label{explaw}
\mu (\theta)  = e^{- \alpha \theta} \, , \; \; \alpha > 0  \quad \mbox{and hence}  \quad   \sigma  =  e^{- \alpha \theta} \, v_x^n  \, .
\tag{H$_{exp}$}
\end{equation}
The model \eqref{shear} for the exponential law \eqref{explaw} admits a special class of solutions 
describing uniform shearing
\begin{equation} \label{ARUSS}
\begin{aligned}
& v_s (x) = x, \\
& \theta_s (t) = \frac{1}{\alpha} \log{\left(\alpha t + c_0\right)}, \quad   c_0 = e^{\alpha\theta_0},\\
& \sigma_s (t)  =  \frac{1}{\alpha t + c_0} \, .
\end{aligned}
\end{equation}
The graph $\sigma_s - t$  describing the average stress vs. average strain response is monotonically decreasing, and thus
the deformation of uniform shear is characterized by strain softening response.
We set the following goals:
\begin{enumerate}
\renewcommand{\theenumi}{\alph{enumi}}
\item[(i)] To study the linearized problem around the uniform shearing 
solution and to assess the effects of the 
various parameters of the problem
\item[(ii)]  To inquire whether there are special solutions of the nonlinear 
problem that indicate the emergence of localized deformations and formation of shear bands.
\end{enumerate}

Since the base solution \eqref{ARUSS} around which we linearize is time-dependent the linearized analysis necessarily 
involves analysis of non-autonomous systems and also perturbations have to be compared against the base solution.
We perform a linearization using ideas of relative perturbations 
\begin{equation}
\label{introlinpert}
\begin{aligned}
u := v_x &= 1 + \delta \bu (x,t)  + O(\delta^2)  ,
\\
\theta
&=  \theta_s (t)  + \delta \btheta (x,t)  + O (\delta^2) ,
\\
\sigma 
&= \sigma_s (t)   + \delta  \sigma_s (t) \bs (x,t)  + O (\delta^2) ,
\end{aligned}
\end{equation}
leading to linearized system
\begin{equation}
\label{introlinrelexp}
\begin{cases}
\bu_t &= \sigma_s  (t) \, \big ( n \bu_{xx} - \alpha \btheta_{xx} \big ),
\\
\btheta_t - \kappa \btheta_{xx} &= \sigma_s  (t) \,  \big ( (n+1) \bu - \alpha \btheta \big ),
\\
\bu_x (0,t) = \bu_x (1,t) &= 0 \, ,
\\
\btheta_x (0,t) = \btheta_x (1,t) &= 0  \, .
\end{cases}
\end{equation}
with $\bs$ computed by $\bs = n \bu - \alpha \btheta$. This problem turns out to have the following behavior:

\begin{enumerate}
\renewcommand{\theenumi}{\alph{enumi}}
\item \label{en:introa}For the case without thermal diffusion $\kappa =0 $:

When $n = 0$ the high frequency modes grow exponentially fast with 
the exponent of the order of the frequency
what is characteristic of ill-posedness and \emph{ Hadamard instability}. 

When $n > 0$ the modes are still unstable but the rate of growth is 
bounded independently of the frequency.
In fact, the behavior in that range is characteristic of \emph{Turing 
instability}.

\item The effect of thermal diffusion ($\kappa > 0$) for the 
non-autonomous linearized model \eqref{introlinrelexp} can also be assessed: 
Perturbations do grow initially in the early stages of deformation,  but over time 
the effect of the diffusion
intensifies and eventually the system gets stabilized. For a quantitative description
of the above qualitative statement we refer to lemmas \ref{lemdecay} and \ref{lembound}. In addition, this behavior
can also occur for the full solution of the nonlinear problem \eqref{introshear} subject to simple shear boundary
conditions as indicated in numerical simulations presented in Figure \ref{mtstb}. This type of metastable response was already conjectured 
(for a somewhat different model) in numerical experiments of \cite{WW, Wa}. The linearized analysis presented here
offers a theoretical explanation of this interesting behavior.
\end{enumerate}

In the second part of this article we consider the adiabatic variant of the model
\eqref{introshear} with exponential constitutive law \eqref{explaw}. The system is expressed for $u=v_x$ 
in the form
\begin{equation}
  \label{introadiashear} 
  \begin{aligned}
    & u_{t} = \del_{xx} \left (  e^{-\alpha \theta} u^n  \right ), 
    \\
    & \theta_{t} =  \ e^{-\alpha \theta} u^{n+1},
    \end{aligned}
\end{equation}
and is now considered on the whole space $(x,t) \in \R \times [0,\infty)$.
For the system \eqref{introadiashear}, we construct a class of localizing solutions 
\begin{equation}
\label{localizsol}
\begin{aligned}
u(x,t) &= \phi_\lambda (t) \, U_\lambda \big ( \sqrt{\lambda} x \phi_\lambda (t) \big ),
\\[2pt]
\sigma (x,t) &= \sigma_s (t) \frac{1}{\phi_\lambda (t)}  \,  \Sigma_\lambda \big ( \sqrt{\lambda} x \phi_\lambda (t) \big ),
\\[2pt]
\theta(x,t) &= \theta_s(t) + \lambda \tfrac{n+1}{\alpha} (\theta_s(t) - \theta_0) + \Theta_\lambda \big ( \sqrt{\lambda} x \phi_\lambda (t) \big ),
\end{aligned}
\end{equation}
where $\displaystyle{\phi_\lambda (t)  = \big  ( \tfrac{\alpha}{c_0} t + 1 \big )^{ \tfrac{\lambda}{\alpha}}}$ depends on the parameter  $\lambda > 0$,
and \eqref{localizsol} is determined  via the solution $(U_\lambda, \Sigma_\lambda, \Theta_\lambda)$
of the singular initial value problem \eqref{msys}, \eqref{mic} with $\nu = \lambda$. 
The existence of the solution $(U_\lambda, \Sigma_\lambda, \Theta_\lambda)$ is established in Theorem \ref{connection} where 
 most importantly its behavior at the origin and at infinity is made precise.
 The emerging solution \eqref{localizsol} depends on three parameters: $\theta_0$ linked to the uniform shear in \eqref{ARUSS},
 a parameter $\lambda$ which can be viewed as a length scale of initial data, and a parameter  $\Sigma_0$ determining the size of the initial 
 profile. 
  
 To our knowledge this is the first instance that the compound effect of Hadamard instability and 
 parabolic regularizations is  mathematically analyzed in a nonlinear context. The reader should contrast the nonlinear response captured by the
 exact solution \eqref{localizsol} to the response of the associated linearized problem \eqref{introlinrelexp}.
The solutions \eqref{localizsol} indicate that the competition among  Hadamard instability 
and  momentum diffusion in a nonlinear context can lead to a nonlinear coherent structure that localizes  in a narrow band (see Theorem \ref{shearband}
and Remark \ref{rmkbeh}).
By contrast, for a linear model, the combined effect of Hadamard instability and parabolicity leads to unstable oscillatory 
modes growing in amplitude  (see \eqref{unstlinear}). This justifies the 
scenario  that nonlinearity can kill the oscillations caused by
Hadamard instability so that the non-uniformity of deformation produced by 
the instability  merges into a single localized zone.
It would be interesting to explore this type of response  for other models related to material instability in solid mechanics.

The exposition is organized as follows: Section \ref{secdescription} presents a description of the problem and the formulation via relative perturbations.
The linearized analysis is presented in section \ref{STABAR}, first the part that can be done via spectral analysis for autonomous systems, 
and then the part concerning the effect of thermal diffusion, which is based on  energy estimates for a non-autonomous linear system. Detailed asymptotics for
the eigenvalues and the connection with Hadamard and Turing instabilities are presented there. The presentation of the localized solutions
is split in  sections \ref{seclocal}-\ref{seclocalII}: Section \ref{seclocal} presents the invariance properties, the {\it ansatz} of localized
solutions, and also introduces the relevant auxiliary problem \eqref{msys}, \eqref{mic}. The construction of solutions 
to  the singular initial-value problem \eqref{msys}, \eqref{mic}
and the proof of their properties is the centerpiece of our analysis and forms the objective of section \ref{secss}. The construction is based on a series of steps,
involving desingularization, a nonlinear transformation, and eventually the existence of a heteroclinic orbit via the Poincar{\'e}-Bendixson theorem.
The construction of the localized solutions is then effected in section \ref{seclocalII}.


\section{Modeling of Shear Bands - Relative perturbations}
\label{secdescription}

In this work we  study the onset of localization and formation of shear bands for 
the model \eqref{introshear} describing shear motions of a non-Newtonian fluid with viscosity
 $\mu (\theta)$ decreasing with temperature.
For concreteness we take a viscosity satisfying
the exponential law \eqref{explaw} and are interested in the range where the rate sensitivity $n$ is positive and small.
The model \eqref{introshear} then becomes
\begin{equation}
  \label{shear} 
  \begin{aligned}
    & v_{t} = \sigma_{x},\\
    & \theta_{t} = \kappa \theta_{ x x}  +  \sigma v_x, \\
    & \sigma  = e^{-\alpha \theta} v_x^n \, .
  \end{aligned}
\end{equation}
and describes shear motions between two parallel plates located at $x=0$ and $x=1$
with $v$ the velocity in the shear direction, $\sigma$ the shear stress,  $\theta$ the temperature, and
\begin{equation}
\label{defsr}
u = v_x,
\end{equation}
depicting the strain rate. 
The boundary conditions are taken
\begin{alignat}{2}
\label{vbc}
v(0,t) &= 0 \, , \quad v(1,t) &&= 1 \, ,
\\
\label{abc}
\theta_x (0,t) &= 0 \, , \quad \theta_x (1,t) &&= 0 \, .
\end{alignat}
Equation \eqref{vbc} reflects prescribed velocities at the  boundaries, while \eqref{abc} manifests
that the bounding plates are adiabatic.
Note that  \eqref{vbc} implies that  the stress satisfies 
\begin{equation}
\label{sbc}
\sigma_x (0,t) = \sigma_x (1,t) = 0 \, ,
\end{equation}
and that
\begin{equation}
\label{ave}
\int_0^1 u(x,t) dx = 1 \, .
\end{equation}
The initial conditions are
\begin{equation}
\begin{aligned}
v(x,0) &= v_0 (x) \, ,
\\
\theta(x,0) &= \theta_0 (x) \, .
\end{aligned}
\end{equation}
Throughout the study we will consider two cases:  When $\kappa = 0$ the process is adiabatic and the resulting model will be the
main object of study, as it is the simplest model proposed for studying the phenomenon of formation of shear bands.
The case $\kappa > 0$ will also be studied in section \ref{STABAR} as a paradigm to assess  the effect of thermal diffusion.

\subsection{Uniform shearing solutions}

The problem \eqref{shear}, \eqref{vbc}, \eqref{abc} admits a class of special solutions describing uniform shear
\begin{align} 
\label{uss}
\begin{aligned}
& v_s = x, \\
& u_s = \del_x v_s = 1, \\
& \theta_s= \theta_s (t), \\
& \sigma_s =  \sigma_s (t) =   \mu (\theta_s (t)),
\end{aligned}
\\
\label{usseq}
\begin{aligned}
\mbox{with $\theta_s$ defined by} \; \;  &\left \{ \; 
\begin{aligned}
\frac{d \theta_s}{dt} &=  \mu (\theta_s) = \sigma_s, \\
\theta_s(0) &= \theta_0.
\end{aligned}
\right . 
\end{aligned}
\end{align}
For the special case of the exponential law \eqref{explaw}, $\theta_s$ and $\sigma_s$ read
\begin{equation} \label{expuss}
\begin{aligned}
& \theta_s= \frac{1}{\alpha} \log{ ( \alpha t + c_0 )}, \quad   c_0 = e^{\alpha\theta_0},\\
& \sigma_s =  \frac{1}{\alpha t + c_0}.
\end{aligned}
\end{equation}
These are universal solutions, {\it i.e.} they hold for all values of the parameters $n$, $\alpha$, $\kappa$ in the model
\eqref{shear} including the limiting elliptic initial value problem when  $n=0$, $\kappa = 0$. 
The graph $\sigma - t$  is viewed as describing the stress vs. average strain response. 
As the graph is  decreasing, it means that  there is always softening response in the course of the deformation.

\subsection{Description of the phenomenon of shear bands}

The  system \eqref{shear} has served as a paradigm for the mathematical study of the phenomenon of shear bands
({\it e.g.} \cite{DH, Tzavaras87, Tzavaras91, BPV}) occurring during the high strain-rate deformations of metals. 
Formation of shear bands is a well known material instability that usually leads to rupture, and has received considerable
attention in the mechanics literature ({\it e.g.} \cite{CDHS,FM,MC,WW,Wr}). The reader is referred to
\cite{KT2} for a description of the problem intended for a mathematically oriented reader that is briefly also outlined below.

It was recognized by Zener and Hollomon \cite{ZH} (see also \cite{SC}) that the effect of the deformation speed is twofold:
(a) To change the deformation conditions from isothermal to nearly adiabatic; under nearly adiabatic conditions the combined
effect of thermal softening and strain hardening tends to produce net softening response. (b) Strain rate has an effect
{\it per se} and induces momentum diffusion that needs to be included in the constitutive
modeling. A theory of thermoviscoplasticity is commonly used to model shear bands and we refer to \cite{SC,Tzavaras87,WW,WrOc,KT} for
various accounts of the modeling aspects.

The basic mechanism for localization is encoded in the behavior of the following simple system
of differential equations
\begin{equation}
  \label{basicsys} 
  \begin{aligned}
    & v_{t} =   \del_x  \big ( \mu (\theta) v_x^n ) ,\\
    & \theta_{t} =   \sigma v_x  = \mu (\theta) v_x^{n+1}, \\
      \end{aligned}
\end{equation}
(which is the adiabatic case $\kappa = 0$ of \eqref{shear}). This model is appropriate for a non-Newtonian fluid
with temperature dependent viscosity, and is also a simplification of the models proposed  in \cite{ZH, SC} to capture
the mechanism for shear band formation in high strain-rate plastic deformations of metals. The argument - adapted to the language of a temperature dependent fluid -
goes as follows: 
as the fluid gets sheared the dissipation produces heat and elevates the 
temperature of the fluid.
The viscosity decreases with the elevation of the temperature and the 
fluid becomes softer and easier to shear.
One eventuality is that the shear deformation distributes uniformly across the material in the
same manner as it does for the usual Newtonian fluids. This arrangement is described by the uniform shearing
solution \eqref{expuss}.  A second mode of response  is suggested when comparing the behavior of the fluid in two spots:
one hot and one cold. Since the viscosity is weaker in the hot spot  and stronger in the cold spots, the amount of 
shear generated in the hot spot will be larger than the shear generated in the cold spot 
and in turn the temperature difference will be intensified.
It is then conceivable that the deformation localizes into a narrow band 
where the material is considerably hotter (and weaker)
 than the surrounding environment. The two possibilities are depicted in Fig. \ref{ShearFlow}.
 \begin{figure}
\centering
\vspace{-0.1cm}
\includegraphics[scale=0.4]{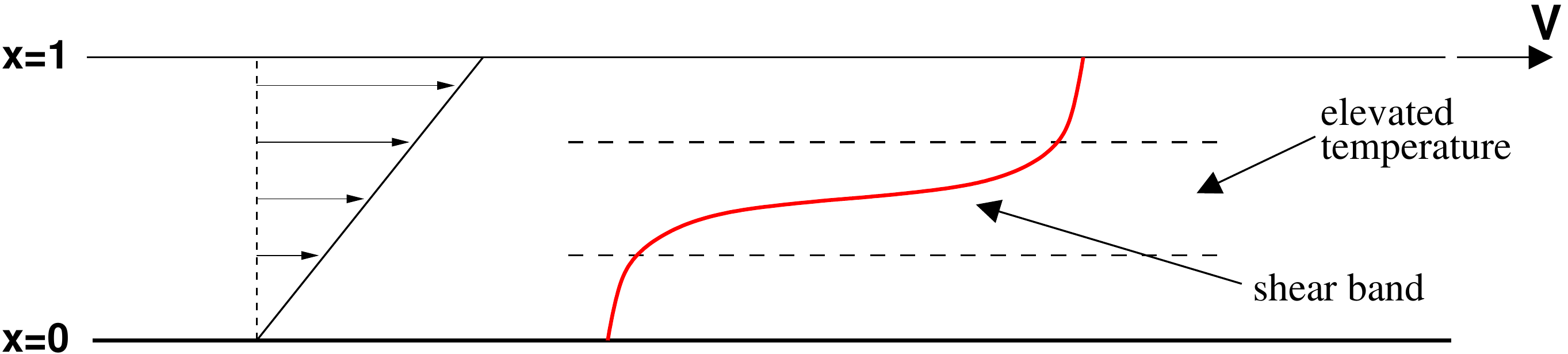}
\vspace{0.1cm}
\caption{Uniform shear versus shear band.}
\label{ShearFlow}
\end{figure}
The outcome of the competition of thermal softening and momentum diffusion
is captured by the behavior of solutions to \eqref{basicsys}.

A series of works analyze the system  \eqref{basicsys} mostly for  power law viscosities 
\begin{equation}
\label{powerlaw}
\mu (\theta)  = \theta^{- \alpha} \, , \; \; \alpha > 0  \quad \mbox{and hence}  \quad   \sigma  =  \theta^{- \alpha} \, v_x^n  \, .
\tag{H$_{pl}$}
\end{equation}
These studies  show stable response  \cite{DH, Tzavaras87} when $-\alpha + n >0$, and indicate unstable response 
\cite{Tzavaras87, BPV, Tzavaras92} but without characterizing the behavior in the complementary region $-\alpha + n < 0$.
More precisely, in the unstable regime Tzavaras \cite{Tzavaras87} shows that development of shear bands can be induced by stress
boundary conditions leading to finite-time blow-up, and that their occurrence is associated with a collapse of stress diffusion across the band.
On the other hand, Bertsch, Peletier and Verduyn-Lunel \cite{BPV} show that if the problem has global solutions and there is a sufficiently
large initial temperature perturbation, then a shear band appears as the asymptotic in time state.
Careful and detailed numerical investigations \cite{WW, ELW, BKT} indicate that instability is widespread in the parameter regime 
$-\alpha + n < 0$ and is typically  followed by the development of shear bands.

The behavior in the ``instability domain" is at present poorly understood, 
especially regarding the precise mechanism of  formation
of the shear band.  The reason is  the response in this regime results from the competition
between an ill-posed problem and a parabolic regularization, and little
mathematical theory is available to study such effects. Nevertheless, an asymptotic theory similar to the Chapman-Enskog expansion
is used by Katsaounis-Tzavaras \cite{KT, KT2} to calculate an
effective equation for the compound effect  of thermal softening and strain-rate sensitivity. The effective equation changes type 
from forward to backward parabolic across the instability threshold, and thus provides an asymptotic criterion for
the onset of localization. In the present work,  we will develop a mathematical framework for studying the competition 
and for examining the conjecture of \cite{KT}.  Our analysis is facilitated by 
the use of the exponential law  \eqref{explaw} inducing  special scaling properties on the resulting system of partial differential equations.


\subsection{Hadamard instability}

In the mechanics literature shear localization is typically associated with  Hadamard-instability. 
This is easily put on display for \eqref{shear}, 
by considering the limiting case $n=0$ and $\kappa = 0$. In the parlance of mechanics, this corresponds to rate-insensitive 
deformations and absence of heat conduction, and leads to the system of conservation laws
\begin{equation}
  \label{hadamard}
  \begin{aligned}
 \del_t 
 \begin{pmatrix}
 v \\ \theta
 \end{pmatrix}
 =
 \begin{pmatrix}
  ( \mu (\theta) )_{x} \\  \mu(\theta)  v_x 
 \end{pmatrix}
 =
    \begin{pmatrix}
      0 &
      \mu'(\theta)
       \\
       \mu(\theta)  & 
	0
    \end{pmatrix}
    \del_x
     \begin{pmatrix}
 v \\ \theta
 \end{pmatrix} 
 .
  \end{aligned}
\end{equation}
For $\mu'(\theta) < 0$ the wave speeds are both imaginary, $\lambda_{\pm} = \pm  \ii \sqrt{ |\mu'| \mu}$, and the 
linearized initial value problem  for \eqref{hadamard} is ill-posed.
Two questions then emerge: 
\begin{enumerate}
\renewcommand{\theenumi}{\roman{enumi}}
\item
 To assess the combined effect of  the competition between  Hadamard 
instability and
 momentum diffusion  and/or  heat conduction (as manifested by small $n > 0$ and $\kappa > 0$ respectively). 
 \item
While Hadamard instability indicates the catastrophic growth of oscillations around a mean state, 
it does not by itself explain the formation of coherent structures typically observed in localization. The latter is
an outcome of the nonlinear effects and has to be assessed at the level of the nonlinear problem.
\end{enumerate}

\subsection{Formulation via relative perturbations}
\label{secrelpert}

The first objective is to develop quantitative criteria for the stability of the uniform shearing solutions. 
Since the basic solutions are time dependent an issue arises how to define their stability. 
Following \cite{MC,Tzavaras92}, 
we  call the uniform shearing solution {\it asymptotically stable} if perturbations of the
solution decay faster than the basic solution, and we call it {\it unstable} if they grow faster than the uniform shearing solution.

It is expedient to introduce a relative perturbation formulation of the problem:
For the {\it exponential law} \eqref{explaw}, we introduce the state $(U, \Theta, \Sigma)$ describing the relative perturbation
and defined through the transformation
\begin{equation}
\label{exptran}
\begin{aligned}
u(x,t) &= U(x,t),
\\
\theta (x,t) &= \theta_s (t)  + \Theta (x,t),
\\
\sigma (x,t) &= \sigma_s (t) \Sigma (x,t),
\end{aligned}
\end{equation}
where $(\theta_s, \sigma_s)$ is given by \eqref{expuss} and satisfies in particular that
\begin{equation}
\label{difequsexp}
\dot\theta_s = \sigma_s = e^{-\alpha \theta_s} \, .
\end{equation}
Using \eqref{shear} and 
\eqref{difequsexp} we derive the system satisfied by the relative perturbation $(U, \Theta, \Sigma)$ in the form
\begin{equation}
\label{relexp}
\begin{aligned}
U_t &= \sigma_s (t) \,  \Sigma_{xx},
\\[2pt]
\Theta_t - \kappa \Theta_{xx} &= \sigma_s (t)   \, \big ( \Sigma \, U -1  \big ),
\\[2pt]
\Sigma &= e^{-\alpha \Theta} U^n.
\end{aligned}
\end{equation}
Note that the uniform shearing solution is mapped to the state $U_0 =1$, $\Theta_0 = 0$ and $\Sigma_0 =1$, and the latter is
an equilibrium for the system \eqref{relexp}.   The problem
then becomes to study the stability of the equilibrium $(U_0, \Theta_0, \Sigma_0) = (1,0,1)$ for the non-autonomous system \eqref{relexp}.

%
%
%
%
%

\section{Linearized analysis of the exponential model}
\label{STABAR}

In this section we present an analysis of the linearized system. Since the uniform shearing solution is
time-dependent the linearized analysis will seek to classify the precise growth in-time of perturbations.

\subsection{The linearized problem}

We compute the linearized problem, using  the ansatz
\begin{equation}
\label{linexp}
\begin{aligned}
U &=  1 + \delta \bu + O(\delta^2),
\\
\Theta &= \delta \btheta + O(\delta^2),
\\
\Sigma &= 1 + \delta \bs + O(\delta^2),
\end{aligned}
\end{equation}
where $\delta$ is a small parameter measuring the size of the perturbation and 
$\delta \ll 1$. 
Expansion of the  equation \eqref{relexp}$_3$ gives to the leading order
of $\delta$
$$
1 + \delta \bs = 1 + (-\alpha \btheta + n \bu)  \delta + O(\delta^2),
$$
while linearizing the other two equations gives
$$
\begin{aligned}
\bu_t &= \sigma_s (t) \, \bs_{xx} + O(\delta^2),
\\
\btheta_t - \kappa \btheta_{xx} &= \sigma_s (t) \,  (\bs + \bu) + O(\delta^2) \, .
\end{aligned}
$$
Collecting the equations together and neglecting the terms of $O(\delta^2)$ we obtain the linearized system:
\begin{equation}
\label{linrelexp}
\left \{ \;
\begin{aligned}
\bu_t &= \sigma_s  (t) \, \big ( n \bu_{xx} - \alpha \btheta_{xx} \big ),
\\[2pt]
\btheta_t - \kappa \btheta_{xx} &= \sigma_s  (t) \,  \big ( (n+1) \bu - \alpha \btheta \big ),
\\[2pt]
\bs &= n \bu - \alpha \btheta .
\end{aligned}
\right .
\end{equation}
The boundary conditions \eqref{abc} and \eqref{sbc} and the transformations \eqref{exptran} and 
\eqref{linexp}, imply that the induced boundary conditions for the linearized system are
\begin{equation}
\label{linbc}
\left \{ \;
\begin{aligned}
\bu_x (0,t) = \bu_x (1,t) &= 0 \, ,
\\
\btheta_x (0.t) = \btheta_x (1,t) &= 0  \, .
\end{aligned}
\right .
\end{equation}

In view of the transformations \eqref{relexp} and \eqref{linexp}, the 
relation between the original
solution $(u , \sigma, \theta)$ and the solution of the linear problem \eqref{linrelexp} is as follows
\begin{equation}
\label{linpert1}
\begin{aligned}
u = v_x &= 1 + \delta \bu (x,t)  + O(\delta^2) , 
\\
\theta
&=\frac{1}{\alpha} \log ( \alpha t + c_0 ) + \delta \btheta (x,t) + O (\delta^2),
\\
\sigma 
&= \frac{1}{\alpha t + c_0 }\big ( 1  + \delta \bs (x,t) + O (\delta^2) \big ) .
\end{aligned}
\end{equation}

\subsection{A change of time scale}
\label{secchatime}

Our analysis of the linearized equation is facilitated by
introducing the transformation of variables
\begin{equation}
\label{hlin}
\begin{aligned}
\bu (x,t) &= \hu (x, \tau (t) ),
\\
\btheta (x,t) &= \htheta (x, \tau(t) ),
\\
\bs (x,t) &= \hs (x, \tau(t)),
\end{aligned}
\end{equation}
where $\tau(t)$ is a rescaling of time, defined by
\begin{equation*}
\left \{
\begin{aligned}
\frac{d\tau}{dt} &= \sigma_s(t)=\frac{1}{\alpha t + c_0}
\\
\tau(0) &= 0
\end{aligned}
\right .
\qquad
\implies  \quad \tau(t)=\frac{1}{\alpha}\log{\left( \frac{c_0 + \alpha  t }{c_0} \right)}.
\end{equation*}
The map $\tau(t)$ is invertible and its inverse is given by
$$
t = \hat t (\tau) = \frac{c_0}{\alpha} \left ( e^{\alpha \tau} -1 \right ) \, .
$$
One computes that $(\hu(x,\tau), \htheta (x,\tau) , \hs (x,\tau) )$ satisfy the linear system
\begin{equation}
\label{linreltau}
\left \{ \;
\begin{aligned}
\frac{\del \hu}{\del \tau} &= \big ( n \hu_{xx} - \alpha \htheta_{xx} \big ),
\\[2pt]
\frac{\del \htheta}{\del \tau}   - \kappa c_0 e^{\alpha \tau}  \htheta_{xx} &=  (n+1) \hu - \alpha \htheta ,
\\[2pt]
\hs &= n \hu - \alpha \htheta ,
\end{aligned}
\right .
\end{equation}
with boundary conditions
\begin{equation}
\label{linrelbc}
\left \{ \;
\begin{aligned}
\hu_x (0,t) = \hu_x (1,t) &= 0 \, ,
\\
\htheta_x (0.t) = \htheta_x (1,t) &= 0  \, .
\end{aligned}
\right .
\end{equation}

For further reference, we record that the original variables $(u, \sigma, \theta)$ and the
solution $(\hu , \hs, \htheta)$ of the rescaled problem \eqref{linreltau}-\eqref{linrelbc}
are related through the formulas
\begin{equation}
\label{linpert2}
\begin{aligned}
u = v_x &= 1 + \delta \hu (x, \tau(t)) + O(\delta^2) , 
\\
\theta
&=\frac{1}{\alpha} \log ( \alpha t + c_0 ) + \delta \htheta (x, \tau(t)) + O (\delta^2),
\\
\sigma 
&= \frac{1}{\alpha t + c_0 }\big ( 1  + \delta \hs (x, \tau(t)) + O (\delta^2) \big ) .
\end{aligned}
\end{equation}

\subsection{Linearized analysis with frozen coefficients}
\label{seclinfroz}

We next consider the system 
\begin{equation}
\label{linref}
\left \{ \;
\begin{aligned}
\frac{\del \hu}{\del \tau} &= \big ( n \hu_{xx} - \alpha \htheta_{xx} \big ) ,
\\[2pt]
\frac{\del \htheta}{\del \tau}   - k \htheta_{xx} &=  (n+1) \hu - \alpha \htheta ,
\end{aligned}
\right .
\end{equation}
with boundary conditions \eqref{linrelbc}. Note that in  \eqref{linreltau} the thermal diffusion coefficient
$k = k(\tau) = \kappa c_0 e^{\alpha \tau}$ is time dependent, while  in \eqref{linref} this coefficient
is frozen in time. We analyze the solutions of \eqref{linref} via an  eigenvalue analysis. This provides the 
following information:
\begin{enumerate}
\renewcommand{\theenumi}{\alph{enumi}}
\item For  $k =0$ it provides the stability or instability properties  
of the linearized system \eqref{linreltau}.
\item For $k > 0$ the eigenvalue analysis  will only offer an 
indication of the effect of thermal diffusion
on the linearized analysis. This will be complemented by an energy estimate analysis of the exact system
 \eqref{linrelexp} in the following section.
 \end{enumerate}

The solutions of \eqref{linref}-\eqref{linrelbc} are expressed as a cosine Fourier series
\begin{equation}
\label{fs}
\begin{aligned}
\hu (x,\tau) &= \sum_{j =0}^\infty \fhu_j (\tau) \, \cos j \pi x ,
\\
\htheta(x,t) &= \sum_{j =0}^\infty \fhtheta_j (\tau) \, \cos j \pi x ,
\end{aligned}
\end{equation}
where the Fourier coefficients $(\fhu_j (\tau), \fhtheta_j (\tau))$ satisfy the  ordinary differential equations
\begin{equation}
\label{syscoef}
\begin{aligned}
\frac{d}{dt}
\begin{pmatrix} 
\fhu_j \\ 
\fhtheta_j
\end{pmatrix}
&=
\begin{pmatrix}
0  & 0 \\
0 & - k (j \pi)^2
\end{pmatrix}
\begin{pmatrix} 
\fhu_j \\ 
\fhtheta_j
\end{pmatrix}
+
\begin{pmatrix}
- n (j \pi)^2  & \alpha (j \pi)^2 \\
n+1 & - \alpha 
\end{pmatrix}
\begin{pmatrix} 
\fhu_j \\ 
\fhtheta_j
\end{pmatrix}
\\
&=
\begin{pmatrix}
- n (j \pi)^2  & \alpha (j \pi)^2 \\
n+1 & - \alpha - k (j \pi)^2
\end{pmatrix}
\begin{pmatrix} 
\fhu_j \\ 
\fhtheta_j
\end{pmatrix}
\end{aligned}
\end{equation}

The eigenvalues of the matrix 
$$
A : = \begin{pmatrix}
- n (j \pi)^2  & \alpha (j \pi)^2 \\
n+1 & - \alpha - k (j \pi)^2
\end{pmatrix} ,
$$
are the roots of the binomial equation
\begin{equation}
\label{binomial}
\lambda^2 + \lambda \big ( \alpha + n (j \pi)^2 + k (j \pi)^2 \big ) + n k (j \pi)^4 - \alpha (j \pi)^2 = 0  \, , 
\quad j = 0, 1, 2, ...
\end{equation}

\bigskip
\noindent
{\bf The mode $j=0$}. For $j= 0$ the equation becomes $\lambda^2 + \alpha \lambda = 0$ and has two eigenvalues
$\lambda_{0, 0} = 0$ and $\lambda_{0, -} = -\alpha$. This corresponds to marginal stability for the linearized system. 
However, any instability associated with the eigenvalue $\lambda_{0, 0} = 0$ is eliminated by \eqref{defsr}, \eqref{vbc} which imply
for the perturbation $\hat u$ in \eqref{linpert2} the condition $\int_0^1 \bu (x, 0) dx = 0$.
Indeed, the system  \eqref{syscoef} for $j =0$ takes the form
$$
\begin{aligned}
\frac{d \fhu_0 }{d\tau} &= 0 ,
\\
\frac{d \fhtheta_0 }{d\tau} &=  (n+1) \fhu_0 - \alpha \fhtheta_0 . 
\end{aligned}
$$
This implies, using  \eqref{fs}, \eqref{linpert1}, \eqref{defsr} and the boundary condition \eqref{vbc}, that
$$
\fhu_0 ( \tau ) = \fhu_0 (0) = \int_0^1 \bu (x, 0) dx = 0.
$$
and clearly $\fhtheta_0 (\tau)$ decays. 
Thus, the $0$-th mode decays.

\bigskip
\noindent
{\bf The modes $j= 1, 2, ... $}.  For these modes the eigenvalues are
\begin{equation}
\label{eigenvj}
\lambda_{j, \pm} =  - \tfrac{1}{2} \big ( \alpha + n (j \pi)^2 + k (j \pi)^2 \big ) \pm \tfrac{1}{2} \sqrt{ \Delta_j }\,  ,
\end{equation}
where the discriminant
$$
\begin{aligned}
\Delta_j &: = \big ( (n + k) (j \pi)^2 + \alpha \big )^2 + 4 \alpha (j \pi)^2 - 4 n k (j \pi)^4
\\
&= (n-k)^2 (j \pi)^4 + \big ( 2 \alpha (n+k) + 4\alpha \big ) (j \pi)^2 + \alpha^2  > 0
\end{aligned}
$$
Since 
$$
\lambda_{j , + } + \lambda_{j , -} = - (n+k) (j \pi)^2 -\alpha < 0 \, , \quad 
\lambda_{j , + } \lambda_{j , -}  = n k (j \pi)^4 - \alpha (j \pi)^2 ,
$$
we conclude that
\begin{equation}
\begin{aligned}
\mbox{ the $j$-th mode is asymptotically stable} \quad &\mbox{iff} \quad nk (j \pi)^2 > \alpha ,
\\
\mbox{ the $j$-th mode is unstable} \quad &\mbox{iff} \quad nk (j \pi)^2 < \alpha .
\end{aligned}
\end{equation}
In summary:
\begin{enumerate}
\renewcommand{\theenumi}{\roman{enumi}}
\item If $k = 0$ then all modes $j=1,2, ...$ are unstable.
\item If $k > 0$ the first few modes might be unstable, that is the 
modes $j$ for which $n k (j \pi)^2 < \alpha$,
but the high frequency modes are stable.
\item The higher $k$ the fewer the number of unstable modes. For $k$ 
 larger than a certain threshold (here equal to $\alpha/(n\pi^2)$)
all modes are asymptotically stable.
\end{enumerate}

\subsection{Asymptotics of the modes and the nature of the instability}

We analyze the behavior of the linearized modes as a function of $j$ and discuss the nature of
the instability in the various special cases of interest. The analysis reveals the role of the parameters $n, k$ 
regarding the nature of the instability.

\subsubsection{ Case 1: $k =0$, $n = 0$, Hadamard instability}

In this regime 
$$
\lambda_{j, \pm} = - \frac{\alpha}{2} \pm \tfrac{1}{2} \sqrt{ \alpha^2 + 4 \alpha (j \pi)^2 } ,
$$
one eigenvalue is  negative and the other strictly positive and increasing with $j$. 
Using the Taylor series development
\begin{equation}
\label{sqtaylor}
\sqrt{ 1 + x} = 1 + \tfrac{1}{2} x - \tfrac{1}{4} x^2 + \tfrac{3}{8} x^3 + ...
\, , \quad \mbox{for $0 < x \ll 1$} ,
\end{equation}
we obtain the following asymptotic developments for the eigenvalues
\begin{equation}
\label{asym1}
\begin{aligned}
\lambda_{j, -} &= - \sqrt{\alpha} j \pi  - \tfrac{\alpha}{2}  - \tfrac{1}{8} \frac{\alpha^{\frac{3}{2}} }{j \pi} + O (\tfrac{1}{j^3} ) ,
\\
\lambda_{j, +} &=  \sqrt{\alpha} j \pi  - \tfrac{\alpha}{2}  + \tfrac{1}{8} \frac{\alpha^{\frac{3}{2}} }{j \pi} + O (\tfrac{1}{j^3} ) \, .
\end{aligned}
\end{equation}

Referring to the formulas \eqref{linpert2}, \eqref{hlin} and \eqref{asym1}, we see that the $j$-th mode of the perturbation grows to the leading
order like
$$
\begin{aligned}
\btheta_j (x,t) = \cos (j \pi x) \, e^{\lambda_{ j, +} \tau(t) }  
&\sim
\cos (j \pi x)   \exp \Big \{ \sqrt{\alpha} j \pi \frac{1}{\alpha}\log{\left( \frac{c_0 + \alpha  t }{c_0} \right)} \Big \}
\\
&=  \cos (j \pi x)  \Big ( \frac{c_0 + \alpha t}{c_0} \Big )^{\tfrac{j\pi}{\sqrt{\alpha}}}\, .
\end{aligned}
$$
Clearly the perturbation grows much faster than the uniform shearing 
solution $\theta_s(t)$ given by \eqref{expuss} where the growth is 
logarithmic,
and the behavior is characteristic of {\it Hadamard instability}, that is the high frequency oscillations  grow at a catastrophic rate.

\subsubsection{Case 2: $k = 0$, $n > 0$, Turing instability}

In this case the eigenvalues of the $j$-th mode are given by the formulas
$$
\lambda_{j, \pm} = - \frac{\alpha + n (j \pi)^2 }{2} \pm \tfrac{1}{2}  (\alpha + n (j \pi)^2 ) \sqrt{ 1 + \frac{ 4 \alpha (j \pi)^2}{  (\alpha + n (j \pi)^2 )^2   } } .
$$
Again we have one negative and one positive eigenvalue and thus unstable 
response in this regime as well.
Using once again the Taylor expansion \eqref{sqtaylor} we obtain the asymptotic expansions   
\begin{equation}
\label{asym2}
\begin{aligned}
\lambda_{j, -} &= - n ( j \pi)^2 - \alpha -   \frac{\alpha (j\pi)^2}{\alpha + n (j\pi)^2} + O (\tfrac{1}{j^2} ) ,
\\
\lambda_{ j, +} &=   \frac{\alpha (j\pi)^2}{\alpha + n (j\pi)^2} + O (\tfrac{1}{j^2} ) \, .
\end{aligned}
\end{equation}
for the positive and negative eigenvalues respectively. 

\begin{lemma}
The positive eigenvalue $\lambda_{j, +}$ satisfies  the following properties:
\begin{align}
\label{prop1}
&\mbox{ $\lambda_{j, +}$ is increasing in $j$} ,
\\
\label{prop2}
&\lambda_{j, +} < \tfrac{\alpha}{n} ,
\\
\label{prop3}
&\lambda_{j, +}  \to \tfrac{\alpha}{n} \quad \mbox{as $j \to \infty$} \ .
\end{align}
\end{lemma}

\begin{proof}
Set $x = (j \pi)^2$ and recall that 
$$
\lambda_+ (x) = - \frac{\alpha + n x}{2} + \tfrac{1}{2} \sqrt{ (nx+\alpha)^2 + 4 \alpha x }\, ,
$$
satisfies the identity
\begin{equation}
\label{idin}
\lambda_+^2 (x) + \lambda_+ (x)  (n x + \alpha) - \alpha x = 0 \, .
\end{equation}
Differentiating \eqref{idin}, we derive
$$
\frac{\del \lambda_+}{\del x} (x) 
=  \frac{ \alpha - \lambda_+ (x) n }{ 2 \lambda_+ (x)  +  (n x + \alpha)}
=  \frac{ \alpha - \lambda_+ (x) n }{ \sqrt{ (n x+\alpha)^2 + 4 \alpha x } } \, .
$$
Next, using the elementary bound $\sqrt{1 + z} < 1 +\tfrac{1}{2} z$ for $z > 0$ we obtain
$$
\begin{aligned}
\lambda_+ (x) &= - \frac{\alpha + n x}{2} + \tfrac{1}{2}  (n x + \alpha)  \sqrt{  1 + \frac{4\alpha x}{ (n x+\alpha)^2 } }
< \frac{\alpha x }{ n x + \alpha} < \frac{\alpha}{n} .
\end{aligned}
$$
This shows that $\frac{\del \lambda_+}{\del x} (x)  > 0$ and provides \eqref{prop1} and \eqref{prop2}. Then \eqref{prop3} follows
from \eqref{asym2}.
\end{proof}

The following remarks are in order: Observe first that, by virtue of \eqref{linpert2}, \eqref{hlin} and \eqref{asym1}, 
the $j$-th mode of the perturbation grows like
\begin{equation}
\label{unstlinear}
\begin{aligned}
\btheta_j (x,t) = \cos (j \pi x) \, e^{\lambda_{j, +} \tau(t) }  
&=
\cos (j \pi x)   \exp \Big \{  \left(    \frac{\alpha (j\pi)^2}{\alpha + n (j\pi)^2} + O (\tfrac{1}{j^2} )     \right )   \tfrac{1}{\alpha}\log{\left( \frac{c_0 + \alpha  t }{c_0} \right)} \Big \}
\\
&\sim  \cos (j \pi x)  \Big ( \frac{c_0 + \alpha t}{c_0} \Big )^{\tfrac{1}{n}  }\ .
\end{aligned}
\end{equation}
Clearly the perturbation grows much faster than the uniform shearing solution $\theta_s(t)$, but the rate of growth is bounded in $j$.
On the other hand, the growth becomes very fast as the rate sensitivity $n \to 0$.

The behavior in this regime is characterized by {\it Turing instability}. Recall, that Turing instability  \cite{Turing}
corresponds to the following scenario that occurs in some instances of dynamical systems. There are two dynamical systems
$$
\dot x = A x \, , \quad \dot x = B x , 
$$
where the spectrum of both  matrices $A$ and $B$ lies in the left-hand 
plane, but  the Trotter
product of the two systems
$$
\dot x = (A + B) x ,
$$
exhibits unstable response.

Indeed,  the dynamical system \eqref{syscoef} with $k = 0$ may be visualized
as the Trotter product of the two systems
\begin{equation}
\label{syscoef1}
\begin{aligned}
\frac{d}{dt}
\begin{pmatrix} 
\fhu_j \\ 
\fhtheta_j
\end{pmatrix}
=
\begin{pmatrix}
- \tfrac{1}{2} n (j \pi)^2  & \alpha (j \pi)^2 \\
0 & - \tfrac{1}{2} \alpha 
\end{pmatrix}
\begin{pmatrix} 
\fhu_j \\ 
\fhtheta_j
\end{pmatrix}
\end{aligned}
\end{equation}
and
\begin{equation}
\label{syscoef2}
\begin{aligned}
\frac{d}{dt}
\begin{pmatrix} 
\fhu_j \\ 
\fhtheta_j
\end{pmatrix}
=
\begin{pmatrix}
- \tfrac{1}{2} n (j \pi)^2  & 0 \\
n+1  & - \tfrac{1}{2} \alpha 
\end{pmatrix}
\begin{pmatrix} 
\fhu_j \\ 
\fhtheta_j
\end{pmatrix}
 .
\end{aligned}
\end{equation}
Each of \eqref{syscoef1} and \eqref{syscoef2} has spectrum in the left half space but the eigenvalues of \eqref{syscoef}
have unstable response.

\subsubsection{Case 3: $k > 0$, $n > 0$, thermal diffusion stabilizes the high modes}

In this regime the eigenvalues are given by \eqref{eigenvj} and they are real and have the properties
\begin{align}
\label{prope1}
&\lambda_{j, -}  < 0 \quad \forall j , 
\\
\label{prope2}
&\lambda_{j, +} < 0  \quad  \Longleftrightarrow  \quad k (j \pi)^2 > \tfrac{\alpha}{n} .
\end{align}
Therefore, the high modes are stable, but the low modes might be unstable if $k$ is sufficiently small.
The ``most unstable" mode is the first mode.
One can again carry out the large $j$ asymptotics of the eigenvalues. The result, in the case
that $n > k$,  reads
\begin{equation}
\label{asym3}
\begin{aligned}
\lambda_{j, -}  &= - n ( j \pi)^2 - \alpha -   \frac{\alpha (k+1)}{ \big |  n - k + \tfrac{\alpha}{ (j\pi)^2 } \big | }  + O (\tfrac{1}{j^2} ) \, ,
\\
\lambda_{j, +} &=  -  k ( j \pi)^2 +  \frac{\alpha (k+1)}{ \big |  n - k + \tfrac{\alpha}{ (j\pi)^2 } \big | }    + O (\tfrac{1}{j^2} ) \, .
\end{aligned}
\end{equation}
A similar formula holds when $n < k$.

\subsection{Effect of thermal diffusion on the linearized problem}

Consider now the linearized system \eqref{linrelexp} with boundary condition $\eqref{linbc}$. Since we are interested in
perturbations of the uniform shearing solution we will consider initial data that satisfy $\int_0^1 u_0 dx =0$ and thus focus
on solutions with
\begin{equation}
\label{average}
\int_0^1 u(x,t) dx = 0 .
\end{equation}
The objective is now to analyze the effect of the time-dependent diffusion on the dynamics of the system.
The linearized analysis of the system with frozen coefficients suggests to expect that diffusion stabilizes the system and its
effect is intensified with the passage of time. Indeed, here we provide a rigorous proof of this statement via energy estimates.

\begin{lemma} 
\label{lemdecay}
There exist $A> 0$ and $T > 0$ such that
\begin{equation}
\label{bound1}
\int_0^1 \big ( \tfrac{A}{2} \bu^2  + \tfrac{1}{2} \btheta^2 \big )(x,t) \, dx \le  \int_0^1 \left ( \tfrac{A}{2} \bu^2  + \tfrac{1}{2} \btheta^2 \right ) (x,T) \, dx
\quad \mbox{for $t \ge T$}
\end{equation}
and
\begin{equation}
\label{bound2}
\int_0^1 \left ( \tfrac{A}{2} \bu^2  + \tfrac{1}{2} \btheta^2 \right )(x,t) dx \to 0 \quad \mbox{as $t \to \infty$}.
\end{equation}
\end{lemma}

\begin{proof}
We multiply \eqref{linrelexp}$_1$ by $\bu$ and \eqref{linrelexp}$_2$ by $\btheta$ and  obtain
the identities
$$
\begin{aligned}
\frac{d}{dt} \int_0^1 \tfrac{1}{2} \bu^2 dx +  n \sigma_s(t) \int_0^1 \bu_x^2 dx &= \alpha \sigma_s (t) \int_0^1 \btheta_x \bu_x dx 
\\
\frac{d}{dt} \int_0^1 \tfrac{1}{2} \btheta^2 dx + \kappa \int_0^1 \btheta_x^2 dx +  \alpha \sigma_s (t) \int_0^1 \btheta^2 dx 
&= (n+1) \sigma_s (t) \int_0^1 \bu \btheta dx 
\end{aligned}
$$
Using Cauchy-Schwarz inequality, gives
\begin{equation}
\label{estim}
\begin{aligned}
\frac{d}{dt} \int_0^1 \tfrac{1}{2} \bu^2 dx +  \tfrac{n}{2} \sigma_s(t) \int_0^1 \bu_x^2 dx &\le \tfrac{\alpha^2}{2n} \sigma_s (t) \int_0^1 \btheta_x^2 dx 
\\
\frac{d}{dt} \int_0^1 \tfrac{1}{2} \btheta^2 dx + \kappa \int_0^1 \btheta_x^2 dx +  \tfrac{\alpha}{2} \sigma_s (t) \int_0^1 \btheta^2 dx 
&\le  \tfrac{(n+1)^2}{2\alpha} \sigma_s (t) \int_0^1 \bu^2 dx  \, .
\end{aligned}
\end{equation}
In view of \eqref{average}, we may use the Poincare inequality
\begin{equation}
\label{poincare}
\int_0^1 \bu^2 dx \le C_p \int_0^1 \bu_x^2 dx \, .
\end{equation}
Combining the above, we note that for $A$ a positive constant we have
$$
\begin{aligned}
\frac{d}{dt} \int_0^1 \left ( \tfrac{A}{2} \bu^2  + \tfrac{1}{2} \btheta^2 \right ) dx &+ \tfrac{A n}{2 C_p} \sigma_s(t) \int_0^1 \bu^2 dx
+  \tfrac{\alpha}{2} \sigma_s (t) \int_0^1 \btheta^2 dx + \kappa \int_0^1 \btheta_x^2 dx
\\
&\le \tfrac{A \alpha^2}{2n} \sigma_s (t) \int_0^1 \btheta_x^2 dx + \tfrac{(n+1)^2}{2\alpha} \sigma_s (t) \int_0^1 \bu^2 dx  \, .
\end{aligned}
$$
Next we select $A$ such that $\tfrac{An}{2 C_p} \ge \tfrac{(n+1)^2}{\alpha}$, and deduce that for $A$ thus selected there is a choice
$c = c(A) > 0$ such that
$$
\begin{aligned}
\frac{d}{dt} \int_0^1 \left ( \tfrac{A}{2} \bu^2  + \tfrac{1}{2} \btheta^2 \right ) dx &+ 
c \sigma_s(t) \int_0^1 \left ( \tfrac{A}{2} \bu^2  + \tfrac{1}{2} \btheta^2 \right ) dx
+ \kappa \int_0^1 \btheta_x^2 dx 
\le \tfrac{A \alpha^2}{2n} \sigma_s (t) \int_0^1 \btheta_x^2 dx  \, .
\end{aligned}
$$

Select now $T$ such that $\tfrac{A \alpha^2}{2n} \sigma_s (t) < \kappa$ for 
$t \ge T$ (recall from \eqref{ARUSS} that $\sigma_s(t)\to 0$ when $t\to 
\infty$). On the interval $t \ge T$, we have the differential
inequality
$$
\frac{d}{dt} \int_0^1 \left ( \tfrac{A}{2} \bu^2  + \tfrac{1}{2} \btheta^2 \right ) dx + 
c \sigma_s(t) \int_0^1 \left ( \tfrac{A}{2} \bu^2  + \tfrac{1}{2} \btheta^2 \right ) dx 
\le 0 \, .
$$
The latter implies \eqref{bound1} and yields, upon using Gronwall's inequality and \eqref{expuss},
$$
\int_0^1 \big ( \tfrac{A}{2} \bu^2  + \tfrac{1}{2} \btheta^2 \big )(x,t) \, dx \le  \Big ( \int_0^1 \left ( \tfrac{A}{2} \bu^2  + \tfrac{1}{2} \btheta^2 \right ) (x,T) \, dx \Big )
\exp \left \{  -c \int_T^t \sigma_s (z) dz \right \}
\to 0  \, , \quad \mbox{as $t \to \infty$}\, ,
$$
and concludes the proof
\end{proof}

The next lemma provides control of the intermediate times.

\begin{lemma} 
\label{lembound}
There exist $B > 0$ and a constant $C_B > 0$ such that 
\begin{equation}
\label{bound3}
\int_0^1 \big ( \tfrac{1}{2} \bu^2  + \tfrac{B}{2} \btheta^2 \big )(x,t) \, dx \le \left ( \int_0^1 \left ( \tfrac{1}{2} \bu^2  + \tfrac{B}{2} \btheta^2 \right ) (x,0) \, dx \right )
\exp \{ C_B t \} 
\quad \mbox{for $t \ge 0$}.
\end{equation}
\end{lemma}

\begin{proof}
Using again \eqref{estim} and \eqref{poincare} we have for any $B >0$
$$
\begin{aligned}
\frac{d}{dt} \int_0^1 \left ( \tfrac{1}{2} \bu^2  + \tfrac{B}{2} \btheta^2 \right ) dx &+ \tfrac{ n}{2 C_p} \sigma_s(t) \int_0^1 \bu^2 dx
+  \tfrac{ B \alpha}{2} \sigma_s (t) \int_0^1 \btheta^2 dx + B \kappa \int_0^1 \btheta_x^2 dx
\\
&\le \tfrac{ \alpha^2}{2n} \sigma_s (t) \int_0^1 \btheta_x^2 dx + B \tfrac{(n+1)^2}{2\alpha} \sigma_s (t) \int_0^1 \bu^2 dx .
\end{aligned}
$$
Select now $B$ large so that $B \kappa > \tfrac{ \alpha^2}{2n} \sigma_s (t)$ for all $t \ge 0$, and then we obtain
$$
\begin{aligned}
\frac{d}{dt} \int_0^1 \left ( \tfrac{1}{2} \bu^2  + \tfrac{B}{2} \btheta^2 \right ) dx 
\le
C_B  \int_0^1 \left ( \tfrac{1}{2} \bu^2  + \tfrac{B}{2} \btheta^2 \right ) dx , 
\end{aligned}
$$
from where \eqref{bound3} follows via Gronwall's inequality.

\end{proof}
We conclude now the analysis of stability. Suppose that for some $\delta > 0$ the initial data satisfy $\int_0^1 u_0 = 0$
and are controlled in $L^2$ by
$$
 \int_0^1 \left ( \tfrac{1}{2} \bu^2  + \tfrac{B}{2} \btheta^2 \right ) (x,0) \, dx \le \delta .
 $$
 Then \eqref{bound1}, \eqref{bound3} imply
 \begin{equation}
 \begin{aligned}
 \int_0^1 \big ( \tfrac{1}{2} \bu^2  + \tfrac{B}{2} \btheta^2 \big )(x,t) \, dx \le \delta e^{C_B t}  \quad \mbox{for $t \ge 0$} , 
 \\
 \int_0^1 \big ( \tfrac{A}{2} \bu^2  + \tfrac{1}{2} \btheta^2 \big )(x,t) \, dx \le  A \delta e^{C_B T}  \quad \mbox{for $t \ge T$} .
 \end{aligned}
 \end{equation}
 Together they imply that the state $(0,0)$ is stable in $L^2$, and in fact on account of \eqref{bound2} it is 
 asymptotically stable. We state the result in the following:
 
 \begin{proposition}
For the initial-boundary value problem \eqref{linrelexp}, \eqref{linbc} with $\kappa > 0$, 
if the initial perturbation $(\bu_0, \btheta_0)$ satisfies
$\int_0^1 \bu_0 dx = 0$ then the equilibrium $(0,0)$ is asymptotically stable in $L^2$.
 \end{proposition}
The proof of the above lemmas  and the analysis of the linearized problem with frozen coefficients indicate that the
perturbation of the base solution may drift far from the origin until eventually it is recaptured as the effect of the 
diffusion intensifies with time.  

This phenomenon of \emph{metastability} also appears in numerical experiments for the nonlinear problem; an instance of these experiments appears in Figure \ref{mtstb}.  We solve numerically
 the rescaled system \eqref{rescaled} for $\kappa=0.5,\ \alpha=0.5, \ n=0.05$. 
 The initial conditions are perturbations of the uniform shearing solutions and they are : $v(x,0)=x$ and $\theta(x,0)=1 + \delta(x)$, where $\delta (x)$ is a small Gaussian perturbation centered at the middle of the interval.  The numerical method is based on adaptivity in space and time. An adaptive finite element method is used for the spatial discretisation while a Runge-Kutta method with variable time step is used for the discretisation in time. Further details concerning the numerical scheme and the adaptivity criteria can be found in \cite{BKT}.   In Figure \ref{mtstb} the evolution in time (in natural logarithmic scale) 
 is presented for the original variables $v$ and $\theta$ in \eqref{resctran}; $\theta$  is sketched in  natural logarithmic scale.  
 We notice that,  after a transient period of localization, thermal diffusion dominates the process, inhomogeneities of the state variables dissipate, 
 and eventually the solution approaches a uniform shearing solution \eqref{uss}. 
 \begin{figure}
\centering
\vspace{-0.1cm}
\includegraphics[width=8cm, height=6cm]{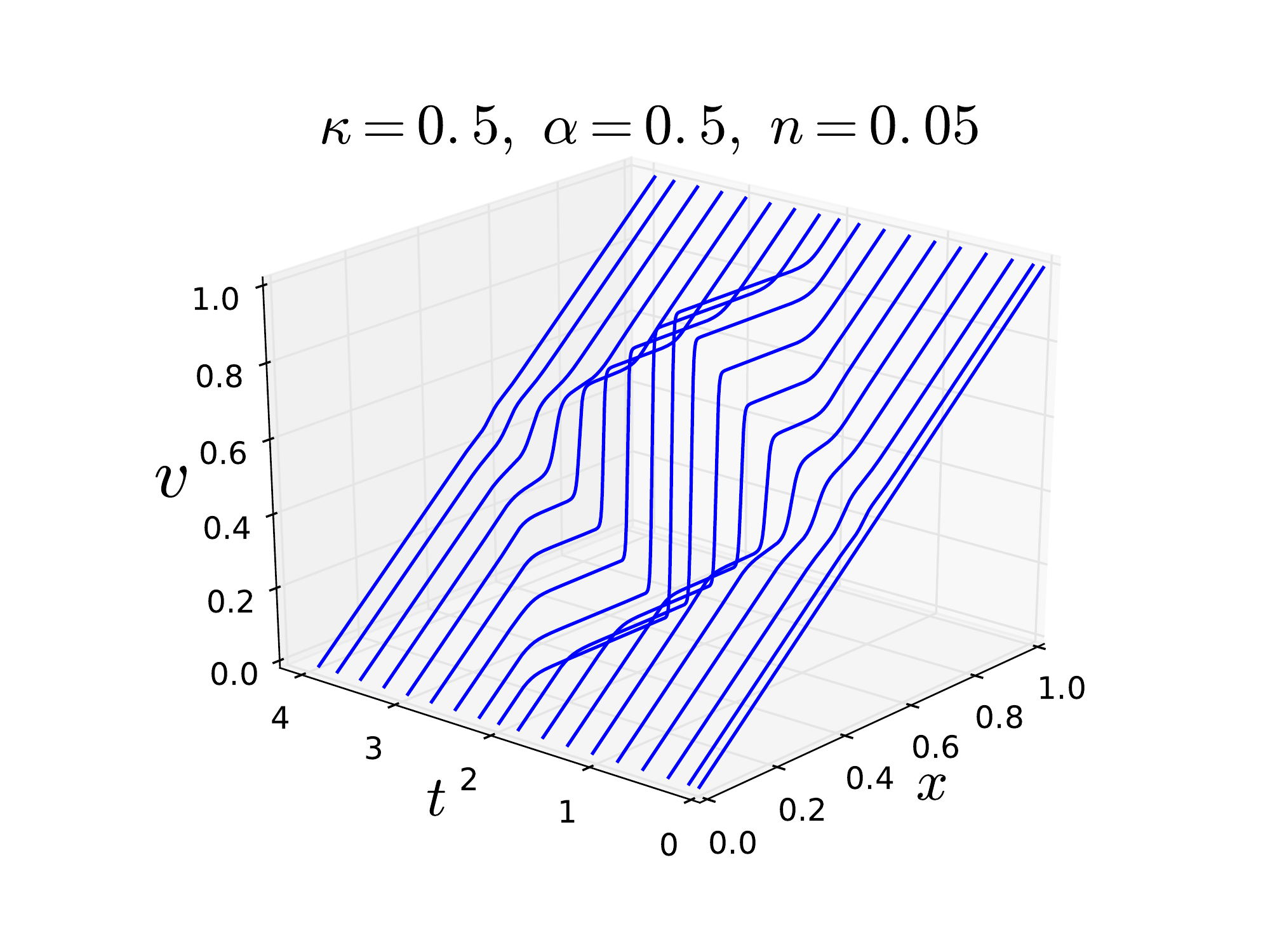}
\includegraphics[width=8cm, height=6cm]{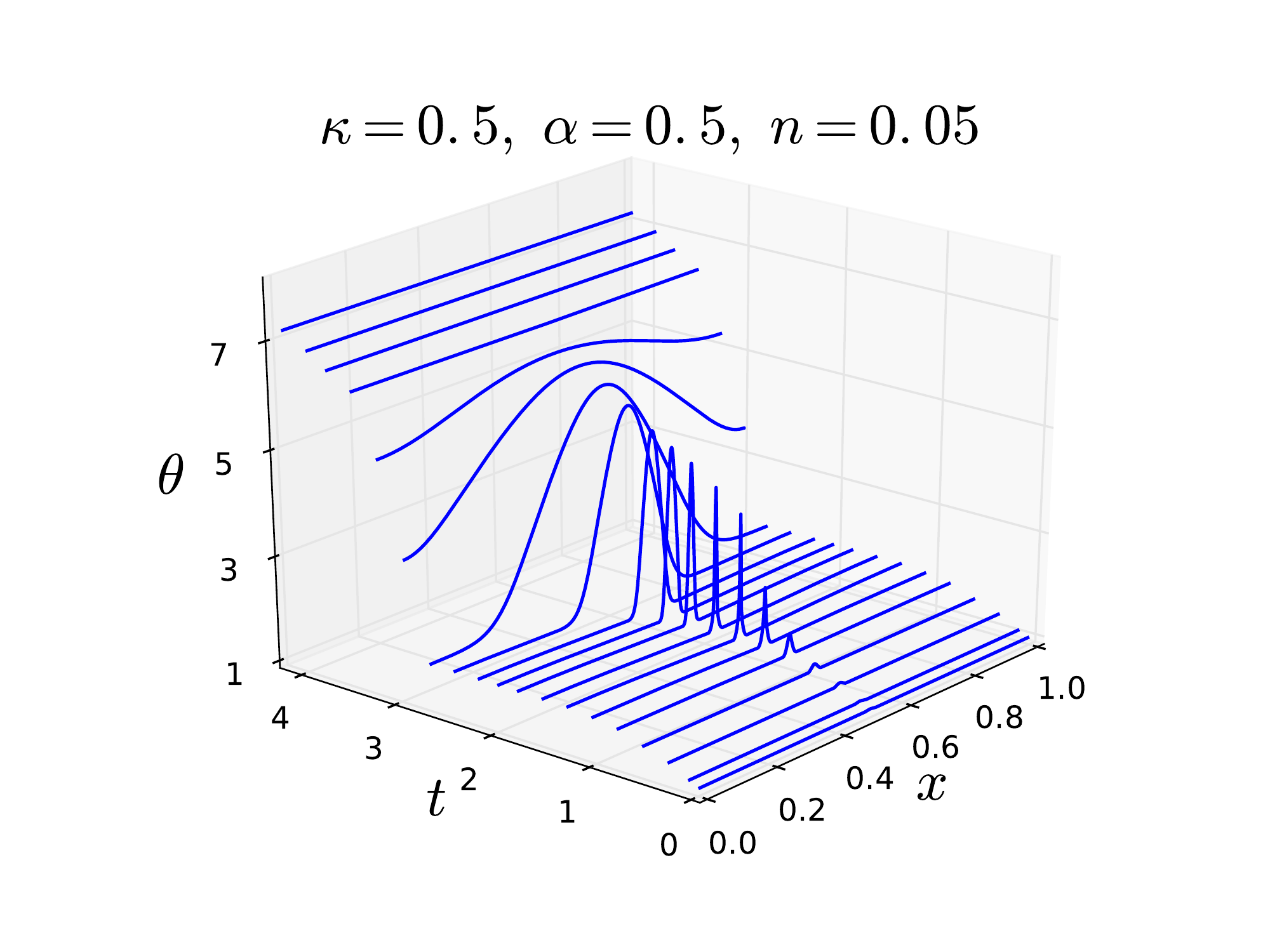}
\vspace{-0.5cm}
\caption{Metastability for the nonlinear problem  : $v(x,t)$ (left), $\theta (x,t)$ (right).}
\label{mtstb}
\end{figure}


\section{Localization - Part I}
\label{seclocal}

We consider next the problem \eqref{shear} on the whole real line and put aside the boundary conditions.
The objective is to construct special solutions that capture a process of localization. This is done over three sections.
The present one is preliminary in nature and discusses various properties that make
the construction feasible and the setting of the problem. The main step is performed in section \ref{secss} where a special class
of self-similar solutions are constructed. The localized solutions are 
then presented in section  \ref{seclocalII}.

\subsection{Relative perturbations and a change of time scale}
We consider \eqref{shear} on $\R \times (0,\infty)$ and recall the notation $u = v_x$. The uniform shearing solutions \eqref{uss}
are defined by solving \eqref{usseq}. Introduce  the transformation
\begin{equation}
\label{resctran}
\begin{aligned}
u(x,t) &= \hU (x, \tau(t) ),
\\
\theta (x,t) &= \theta_s (t)  + \hTheta (x, \tau (t) ),
\\
\sigma (x,t) &= \sigma_s (t) \hS (x,\tau (t) ),
\end{aligned}
\end{equation}
which involves a passage to relative perturbations (as in section \ref{secrelpert}) and a change of time scale from the original to
a rescaled time $\tau(t)$. The scaling $\tau(t)$ is  precisely the one used in the linearized analysis of sections \ref{secchatime} and \ref{seclinfroz};
it is defined by
\begin{equation*}
\tau(t)=\frac{1}{\alpha}\log{ \left( \frac{c_0 + \alpha  t }{c_0} \right ) } = \theta_s (t) - \theta_s (0) \, ,
\end{equation*}
it satisfies
\begin{equation*}
\left \{
\begin{aligned}
\frac{d\tau}{dt} &=  \frac{d\theta_s }{dt} = \sigma_s(t)= e^{-\alpha \theta_s}
\\
\tau(0) &= 0
\end{aligned}
\right .
\end{equation*}
where $\theta_s (t)$ is as in \eqref{expuss}, $c_0 = e^{\alpha \theta_0}$. The scaling $\tau (t)$ satisfies
$\tau (t) \to \infty$ as $t \to \infty$,  and
$\tau(t)$ is invertible with inverse map given by
$$
t = \hat t (\tau) = \frac{c_0}{\alpha} \left ( e^{\alpha \tau} -1 \right ) \, .
$$
The functions $(\hU (x,\tau), \hTheta (x, \tau) , \hS (x, \tau) )$ satisfy the system of partial differential equations
\begin{equation}
\label{rescaled}
\begin{aligned}
\hU_\tau &=  \hS_{xx},
\\[2pt]
\hTheta_\tau - \kappa c_0 e^{\alpha \tau}  \hTheta_{xx} &=  \big ( \hS \, \hU -1  \big ),
\\[2pt]
\hS &= e^{-\alpha \hTheta} \hU^n .
\end{aligned}
\end{equation}
We emphasize that
\begin{itemize}
\item For $\kappa > 0$ the system is non-autonomous, which is natural in 
view of the nature of the transformation \eqref{resctran}.
\item Notably, for $\kappa = 0$ the system \eqref{rescaled} becomes autonomous.
\end{itemize}

\subsection{Effective equation for the adiabatic system}

Next, we consider the adiabatic case $\kappa = 0$ where \eqref{rescaled} takes the form
\begin{equation}
\label{adia}
\begin{aligned}
\hU_\tau &=  \hS_{xx},
\\[2pt]
\hTheta_\tau  &=  \big ( \hS \, \hU -1  \big ),
\\[2pt]
\hS &= e^{-\alpha \hTheta} \hU^n \, .
\end{aligned}
\end{equation}

To understand the long time response of \eqref{adia} we perform a formal asymptotic analysis following ideas from \cite{KT}.
Fix an observational time scale $T$ and introduce the change of variables
$$
\begin{aligned}
\hU(x,\tau) &= \tU_T \left ( \frac{x}{\sqrt{T}} , \frac{\tau}{T}  \right ),
\\
\hS (x,\tau) &= \tS_T \left ( \frac{x}{\sqrt{T}} , \frac{\tau}{T}  \right ),
\\
\hTheta (x,\tau) &= \tTheta_T \left ( \frac{x}{\sqrt{T}} , \frac{\tau}{T}  \right ).
\end{aligned}
$$
One easily checks that $(\tU_T (y,s), \tS_T (y,s) , \tTheta_T (y,s) )$ satisfies the system
\begin{equation}
\label{adiaT}
\begin{aligned}
\tU_s &=  \tS_{y y}\,,
\\[2pt]
\tfrac{1}{T}  \, \tTheta_s  &=  \big ( \tS \, \tU -1  \big ),
\\[2pt]
\tS &= e^{-\alpha \tTheta} \tU^n .
\end{aligned}
\end{equation}
We are interested in calculating the equation describing the effective response of \eqref{adiaT} for $T$ sufficiently large.
This asymptotic analysis involves a parabolic equation \eqref{adiaT}$_1$,  which can be thought as a moment equation,
and equations \eqref{adiaT}$_2$ and \eqref{adiaT}$_3$, which may be thought as describing a relaxation process towards the equilibrium curve
described by 
$$
\begin{aligned}
\tS \tU - 1 &= 0
\\
\tS &= e^{-\alpha \tTheta} \tU^n
\end{aligned}
\quad \Longleftrightarrow  \quad
\tS = \frac{1}{\tU} \, , \quad \tTheta = \frac{n+1}{\alpha} \log \tU \, .
$$
Note that the equilibria are a one-dimensional curve parametrized by $\tU$.

To calculate the effective equation we perform a procedure analogous to the Chapman-Enskog expansion of kinetic theory. 
We refer to \cite[Sec 5]{KT2} for the details of the asymptotic analysis. The procedure produces
the following effective equation
\begin{equation}
\label{effec}
u_s = \del_{yy} \left ( \frac{1}{u} \right ) + \tfrac{1}{T} \tfrac{n+1}{\alpha} \del_{yy} \left ( \frac{1}{u^2} \del_{yy} \frac{1}{u} \right ) \, ,
\end{equation}
which captures the effective response up to order $O (\tfrac{1}{T^2} )$. The leading term is backward parabolic and the first
order correction is fourth order. One checks that the effect of the fourth order term is to stabilize the equation, in the sense that the linearized
equation around the equilibrium $u_0 =1$  is stable (see \cite{KT2} for the details). 

It was further shown in \cite{KT2} that \eqref{effec} admits special 
solutions that exhibit localization, which is not surprising 
as the leading response in \eqref{effec} is backward parabolic. Localized solutions of a similar nature will be  constructed here directly for
the system \eqref{adia}. This is a more ``unexpected" behavior, because  
\eqref{adia} is at least {\it pro-forma} a ``hyperbolic-parabolic" system, 
and the explanation for this behavior lies precisely in the asymptotic 
relation between  \eqref{adia} and
\eqref{effec}.

\subsection{Invariance properties and a related ansatz}
\label{secinvariance}

It is easy to check that  for $\kappa = 0$ the system \eqref{rescaled}  is invariant under the one-parameter scaling transformation
\begin{equation}
\label{scaleinv}
\begin{aligned}
\hU_a  (x, \tau) &= a \hU (a x, \tau) ,
\\
\hS_a (x, \tau)  &= \frac{1}{a} \hS (a x, \tau),
\\
\hTheta_a (x, \tau)  &= \frac{n+1}{\alpha} \log a + \hTheta (a x, \tau) ,
\end{aligned}
\end{equation}
for any $a > 0$. It is also easy to check that this is the only scaling transformation under which the problem is invariant.
This transformation has the distinctive feature that the scaling of space is independent of the scaling of time.
Finally, the system \eqref{rescaled} is not scaling invariant when the
thermal diffusion $\kappa$ is positive.

The scaling invariance \eqref{scaleinv} motivates a change of variables 
and a related {\it ansatz} of solutions. Introduce a variable in
time scaling $r (\tau)$, the variable $\xi = \frac{x}{ r(\tau)}$, and the change of variables
\begin{equation}
\label{adavar}
\begin{aligned}
\hU (x, \tau) &= \frac{1}{r(\tau)} \bU \Big ( \frac{x}{r(\tau) } , \tau \Big ),
\\
\hS (x, \tau)  &= r(\tau) \, \bS \Big ( \frac{x}{r(\tau) } , \tau \Big ),
\\
\hTheta (x, \tau)  &= - \frac{n+1}{\alpha} \log r(\tau)  + \bTheta \Big ( \frac{x}{r(\tau) } , \tau \Big ) .
\end{aligned}
\end{equation}
We denote by $\dot r = \frac{d}{d\tau} r(\tau)$ and by $'=\frac{d}{d\xi}$. Introducing \eqref{adavar} to \eqref{rescaled}, we obtain
that the function $( \bU (\xi, \tau), \bS (\xi, \tau), \bTheta (\xi, \tau) )$ satisfies the system of partial differential equations
\begin{equation}
\label{sysxitau}
\begin{aligned}
\bU_\tau - \frac{\dot r}{r} \big ( \xi \bU \big )_{\xi} &=  \bS_{\xi \xi} ,
\\[2pt]
\bTheta_\tau - \frac{\dot r}{r} \left ( \tfrac{n+1}{\alpha} + \xi \bTheta_{\xi} \right ) ,
-  \kappa c_0 \frac{e^{\alpha \tau}}{r^2 (\tau)}   \bTheta_{\xi \xi} &=  \big ( \bS \, \bU -1  \big ) ,
\\[2pt]
\bS &= e^{-\alpha \bTheta} \bU^n .
\end{aligned}
\end{equation}

The resulting system \eqref{sysxitau} depends on the choice of the scaling function $r(\tau)$. We will be interested in solutions
that are steady states in the new variables. Of course such solutions are dynamically evolving in
the original variables. There are two cases that produce simplified results of interest:
\begin{enumerate}
\renewcommand{\theenumi}{\alph{enumi}}
\item Consider the case $\kappa > 0$, and select $r(\tau) = e^{ 
\tfrac{1}{2} \alpha \tau} $. Then the system \eqref{sysxitau}
is consistent with the ansatz of steady state solutions $(\bU(\xi), \bS(\xi), \bTheta (\xi))$ and the latter are constructed by solving the
system of ordinary differential equations
\begin{equation}
\label{sysxikappa}
\begin{aligned}
 - \frac{\alpha}{2} \big ( \xi \bU \big )' &=  \bS'' ,
\\[2pt]
-\frac{\alpha}{2} \left ( \tfrac{n+1}{\alpha} + \xi \bTheta' \right ) 
-  \kappa c_0   \bTheta''  &=  \bS \, \bU -1 ,
\\[2pt]
\bS &= e^{-\alpha \bTheta} \bU^n .
\end{aligned}
\end{equation}

Since  $r(\tau)  \to \infty$ as $\tau$ increases,  solutions of \eqref{sysxikappa} will produce 
defocusing solutions for \eqref{sysxitau}. A study of \eqref{sysxikappa} is expected to
give information about the convergence to the uniform  shearing state in the non adiabatic setting.
This is not the focus of the present work and we will not deal further with the system \eqref{sysxikappa}.

\item
Instead we consider the case $\kappa = 0$. We may now select any exponential function for $r(\tau)$ but are most interested in
$r(\tau ) = e^{- \lambda \tau}$ with $\lambda > 0$ as this produces focusing solutions. With this selection, 
$(\bU(\xi), \bS(\xi), \bTheta (\xi))$ is a steady solution for \eqref{sysxitau} with $\kappa = 0$ provided it satisfies
\begin{equation}
\label{sysxi}
\begin{aligned}
\bS'' &=  \lambda  \big ( \xi \bU \big )' ,
\\[2pt]
\lambda \left ( \tfrac{n+1}{\alpha} + \xi \bTheta' \right ) 
  &=  \bS \, \bU -1 ,
\\[2pt]
\bS &= e^{-\alpha \bTheta} \bU^n ,
\end{aligned}
\end{equation}
where $\lambda > 0$ is a parameter.
\end{enumerate}

\begin{remark} \rm Note that:
\begin{itemize}
\item
The form of the transformation \eqref{adavar} is motivated by the scaling invariance \eqref{scaleinv}.
\item
The scaling connects the two variables
and produces interesting systems even in the case  $\kappa$ positive when 
the system \eqref{rescaled} is not scale
invariant.
\end{itemize}
\end{remark}

\section{The auxiliary problem}
\label{secss}

In this section we consider the system of ordinary differential equations
\begin{equation}
\label{msys}
\begin{aligned}
\Sigma' &=  \xi U ,
\\[2pt]
\nu \left ( \tfrac{n+1}{\alpha} + \xi \Theta' \right ) 
  &=  \Sigma \, U -1 ,
\\[2pt]
\Sigma &= e^{-\alpha \Theta} U^n \, ,
\end{aligned}
\end{equation}
on the domain $0< \xi < \infty$ with $\nu > 0$,  subject to the initial condition
\begin{equation}
\label{mic}
\Sigma (0) = \Sigma_0>0.
\end{equation}
This is an auxiliary problem that will be studied in 
detail in this section.
We emphasize that it does not coincide with the integrated form of \eqref{sysxi}, however it will be used
together with additional properties  in order to construct solutions of  \eqref{sysxi} 
and localizing solutions for \eqref{adia} in the following section.

In this section, we are interested in constructing solutions $(U (\xi), \Sigma (\xi) , \Theta (\xi))$ for 
the singular initial value problem \eqref{msys}, \eqref{mic} for any $\nu > 0$.  The following elementary remarks are in order:
\begin{enumerate}
\renewcommand{\theenumi}{\alph{enumi}}
\item
When $\nu = 0$,  \eqref{msys} simplifies to 
\begin{equation}
\label{msys0}
\left \{ 
\begin{aligned}
\Sigma' &=  \xi U 
\\[2pt]
\Sigma (0) &= \Sigma_0 
\end{aligned}
 \right .
 \qquad
 \left \{
 \begin{aligned} 
 \Sigma \, U  &= 1 
\\[2pt]
\Theta  &= \frac{n+1}{\alpha} \log U 
\end{aligned}
\right . .
\end{equation}
The solution of \eqref{msys0} is computed explicitly by
\begin{equation}
\label{inn}
\begin{aligned}
\Sigma (\xi) &= \sqrt{  \xi^2 + \Sigma_0^2} ,
\\
U(\xi) &= \frac{1}{ \Sigma (\xi)} = \frac{1}{  \sqrt{\xi^2 + \Sigma_0^2} },
\\
\Theta (\xi) &=  - \frac{n+1}{\alpha} \log \sqrt{  \xi^2 + \Sigma_0^2}  \, .
\end{aligned}
\end{equation}

\item
When $\nu>0$, the system  \eqref{msys} is invariant under the family of 
scaling transformations
\begin{equation}
\label{scaleinv2}
\begin{aligned}
U_b (\xi) &= b U (b \xi) ,
\\
\Sigma_b (\xi)  &= \frac{1}{b} \Sigma (b \xi),
\\
\Theta_b (\xi)  &= \frac{n+1}{\alpha} \log b + \Theta (b \xi),
\end{aligned}
\end{equation}
for  $b > 0$.  There is a  special solution of \eqref{msys} that is self-similar with respect
to the scaling invariance \eqref{scaleinv2} and reads
\begin{equation}
\label{out}
\hS (\xi) = \xi \, , \quad \hU (\xi) = \frac{1}{\xi} \, , \quad \hTheta (\xi) = - \tfrac{n+1}{\alpha} \log \xi  \, .
\end{equation}
This solution does not however satisfy the initial condition \eqref{mic}.
\end{enumerate}

We next construct a function $(U (\xi), \Sigma (\xi) , \Theta (\xi))$ which satisfies  \eqref{msys}, \eqref{mic}
for all values of $\nu > 0$ and $\Sigma_0 > 0$. We will show that the 
constructed solution will have the inner behavior of \eqref{inn} and the 
outer 
behavior of \eqref{out}.

\begin{theorem}
\label{connection}
Let $\Sigma_0 > 0$ be fixed. Let us note $c_\nu=1+\nu(n+1)/\alpha$.
Given $\nu > 0$ there exists a solution $(\Sigma_\nu , U_\nu , \Theta_\nu )$ of \eqref{msys}, \eqref{mic} 
defined for $\xi \in (0, \infty)$ and satisfying the  properties:

\begin{enumerate}
\renewcommand{\theenumi}{\roman{enumi}}

\item It achieves the initial data
\begin{equation}
\label{behavior0}
\lim_{\xi \to 0}
\begin{pmatrix} 
\Sigma (\xi) \\ 
U(\xi)  \\
\Theta(\xi)
\end{pmatrix}
=
\begin{pmatrix}
\Sigma_0 \\
U_0  \\
\Theta_0
\end{pmatrix}
\, , 
\end{equation}
where  $\Sigma_0$ as in \eqref{mic} and $U_0$, $\Theta_0$ are defined via
\begin{equation}
\label{defutheta}
\begin{aligned}
U_0 \Sigma_0 &= c_\nu .
\\
\Theta_0 &= \tfrac{n+1}{\alpha} \log U_0  - \tfrac{1}{\alpha} \log c_\nu \, ,
\end{aligned}
\end{equation}

\item  As $\xi \to \infty$ it has the limiting behavior 
\begin{equation}
\label{behavinfi}
\begin{aligned}
U(\xi)  &= \frac{1}{\xi} + o \left ( \frac{1}{\xi} \right ) ,
\\
\Sigma (\xi) &= \xi + o (\xi) ,
\\[2pt]
\Theta (\xi) &=   - \tfrac{n+1}{\alpha} \log \xi + o(1) .
\end{aligned}
\end{equation}

\item   The solution has the regularity $U$, $\Theta \in C^1 [0,\infty)$, 
$\Sigma \in C^2 [0,\infty)$. It satisfies
\begin{equation}
\label{limitderiv}
\frac{d \Sigma }{d\xi} (0)  =   \frac{dU}{d\xi} (0)  = \frac{d \Theta }{d\xi} (0) = 0 \, , 
\end{equation}
and has the Taylor expansion as $\xi \to 0$, 
\begin{equation}
\label{taylor}
\begin{aligned}
U (\xi) &= U_0  + o (\xi),
\\
\Sigma (\xi) &= \Sigma_0 + \tfrac{1}{2} U_0 \xi^2 +  o (\xi^2),
\\
\Theta (\xi) &= \Theta_0 + o (\xi) \, ,
\end{aligned}
\end{equation}
where as usual $\lim_{\xi \to 0} \tfrac{1}{\xi} o (\xi) = 0$ (but the rate 
may degenerate as $\nu$ tends to $0$).
\end{enumerate}

\end{theorem}

\begin{proof}
{\it Step 1.} Motivated by the special solution \eqref{out} we introduce a change of variables in two steps:
\begin{alignat}{2}
U (\xi) &= \frac{1}{\xi} {\bar u} (\xi) &&=  \frac{1}{\xi} u ( \log \xi) ,
\nonumber
\\
\label{indep}
\Sigma (\xi) &= \xi {\bar \sigma} (\xi)  &&=   \xi \sigma (\log \xi)  ,
\\[3pt]
\Theta (\xi) &= - \tfrac{n+1}{\alpha} \log \xi  +{\bar \theta} (\xi) &&=  - \tfrac{n+1}{\alpha} \log \xi  + \theta (\log \xi),
\nonumber
\end{alignat}
Indeed, we easily  check that $( {\bar u} , {\bar \sigma} , {\bar \theta} )(\xi)$ satisfies the system of equations
\begin{equation}
\label{modmsys}
\begin{aligned}
\xi \, {\bar \sigma}' &=  {\bar u} - {\bar \sigma} ,
\\[2pt]
\nu  \xi \, {\bar \theta}'  &=  {\bar \sigma} \, {\bar u} -1 ,
\\[2pt]
{\bar \sigma} &= e^{-\alpha {\bar \theta}} {\bar u}^n \, .
\end{aligned}
\end{equation}
In the second step, we introduce in \eqref{modmsys} the change of independent variable $\eta = \log \xi$
and derive that $(u, \sigma , \theta)(\eta)$ satisfy 
\begin{equation}
\label{finmsys}
\begin{aligned}
\frac{d\sigma}{d\eta}  &=  u - \sigma,
\\[2pt]
\nu  \frac{d\theta}{d\eta}   &=  \sigma \, u  -1 ,
\\[2pt]
\sigma  &= e^{-\alpha \theta} u^n \, .
\end{aligned}
\end{equation}
The system \eqref{finmsys} is far simpler than the original \eqref{msys} in 
that it is autonomous and non singular.

\medskip
{\it Step 2.} We next consider \eqref{finmsys} and perform a further 
simplification, via the change of dependent variables
\begin{equation}
\label{changed}
\begin{aligned}
a &:= \frac{1}{\sigma},
\\
b &:= \frac{1}{\sigma^{1+\frac{1}{n} } \,  e^{\frac{\alpha}{n} \theta} 
}  = \frac{1}{\sigma u} \, .
\end{aligned}
\end{equation}
Note that $u = \frac{a}{b}$. We now compute
$$
\frac{da}{d\eta} = - \frac{1}{\sigma^2} \frac{d\sigma}{d\eta} 
= - \frac{1}{\sigma^2} \Big ( \sigma^{\frac{1}{n} }  e^{\frac{\alpha}{n} \theta} - \sigma \Big )
= a - \frac{a^3}{b}
$$
and
$$
\begin{aligned}
\frac{db}{d\eta} &= - \big ( 1+\tfrac{1}{n} \big ) b \frac{1}{\sigma} 
\frac{d\sigma}{d\eta} - \tfrac{\alpha}{n} b \frac{d\theta}{d\eta}
\\
&= - \tfrac{n+1}{n} b \frac{1}{\sigma} \left ( \frac{1}{\sigma 
b} - \sigma \right ) - \tfrac{\alpha}{\nu n} b \left ( 
\frac{1}{b} -1 \right )
\\
&= \tfrac{\alpha}{\nu n} \Big  ( \big (1 + \nu \tfrac{n+1}{\alpha} \big )
b -1 - \tfrac{n+1}{\alpha} \nu a^2  \Big ) .
\end{aligned}
$$

In summary, we conclude that $(a , b)(\eta)$ satisfies the autonomous 
system of ordinary differential equations
\begin{equation}
\label{final}
\frac{d}{d\eta}
\begin{pmatrix}
a
\\
b
\end{pmatrix}
=
\begin{pmatrix}
a \big ( 1 - \frac{a^2}{b} \big )
\\
\tfrac{\alpha}{\nu n} \Big  ( c_\nu  b -1 - \tfrac{n+1}{\alpha} \nu 
a^2  \Big ) 
\end{pmatrix}
\;
= :  \; F(a, b) ,
\end{equation}
where 
\begin{equation}
\label{defce}
c_\nu : =  \big (1 + \nu \tfrac{n+1}{\alpha} \big )   \; >  \; 1 \, .
\end{equation}

There are two advantages in \eqref{final} relative to \eqref{finmsys}:  
\begin{enumerate}
\renewcommand{\theenumi}{\roman{enumi}}
\item 
\eqref{final} is expressed via only two independent variables $(a, b)$.
\item The equilibrium at infinity in \eqref{finmsys} is pulled at the axis 
for \eqref{final} via the transformation \eqref{changed}.
\end{enumerate}

\medskip
{\it Step 3.} We perform a phase space analysis of \eqref{final} and 
establish the existence of a heteroclinic connection. We show:

\begin{proposition}
\label{thode}
Consider the  autonomous planar system of differential equations \eqref{final} with $c_\nu$ defined by \eqref{defce}.
For any $\nu > 0$ the system has the following properties:
\begin{itemize}
\item $F$ has  two equilibria $P=(0,1/ c_\nu )$, $Q=(1,1)$ in the first quadrant. Both $P$ and $Q$ are hyperbolic points, $P$ is a repelling node 
while $Q$ is a saddle point.
\item There exists a heteroclinic orbit connecting $P$ to $Q$.
\end{itemize}
\end{proposition}

\medskip\noindent
{\it Proof of Proposition \ref{thode}.} \; 
\begin{enumerate}
\renewcommand{\theenumi}{\alph{enumi}}
 \item 
 The system \eqref{final} has the special solution 
$$
\begin{aligned}
&a (\eta)  = 0 ,
\\
&b(\eta)  \quad \mbox{solves the differential equation } \quad 
\frac{db}{d\eta} = 
\tfrac{\alpha}{\nu n} \big  ( c_\nu  b -1  \big ) .
\end{aligned}
$$
This solution splits the phase portrait into the left and right plane and 
acts as a barrier from crossing
from the one side to the other.

\item The equilibria of \eqref{final} are computed by solving
$$
\begin{aligned}
a \big ( 1 - \frac{a^2}{b} \big ) &= 0 ,
\\
c_\nu  b -1 - \tfrac{n+1}{\alpha} \nu a^2 &= 0 ,
\end{aligned}
$$
and they are $P = ( 0,   \tfrac{1}{c_\nu} )$ and $Q  = (1, 1)$ and $R  =  
(-1, 1)$. Focusing on the behavior
in the first quadrant we neglect the equilibrium $R$.

\item A  computation shows
\begin{equation}
\label{gradient}
\nabla F(a, b)=
\begin{pmatrix}
1 - \frac{3a^2}{b} &  \frac{a^3}{b^2} \\
\\[2pt]
-2 \tfrac{n+1}{n} a &   \tfrac{\alpha}{n\nu} c_\nu
\end{pmatrix}
 .
\end{equation}
At the equilibrium $P$ we have
$$
\nabla F \big (0 , \tfrac{1}{c_\nu} \big )=
\begin{pmatrix}
1 &  0 \\
0 &   \tfrac{\alpha}{n\nu} c_\nu
\end{pmatrix} ,
$$
and we have eigenvalues $\lambda_1 = 1$ with eigenvector $r_1 = (1, 0)^T$, and $\lambda_2 = \tfrac{\alpha}{n\nu} c_\nu > 0$
with eigenvector $r_2 = (0,1)^T$. In view of \eqref{defce}, we have 
$\lambda_2>1$. The local behavior of solutions of \eqref{final} around the 
equilibrium $P$ is that of a repelling node.
 Standard phase plane theory for second order systems ({\it e.g.} \cite{hartman}) states that there is a small neighbourhood of the equilibrium $P$ 
 such that for any given solution
$(a(\eta), b(\eta))$ of  \eqref{final} there are constants 
$\kappa_1$, $\kappa_2$ such that 
\begin{equation}
\label{asymptotic}
\begin{pmatrix}
a(\eta)
\\
b(\eta)
\end{pmatrix}
=
\begin{pmatrix}
0
\\
\tfrac{1}{c_\nu}
\end{pmatrix}
+
\kappa_1 
\begin{pmatrix}  1  \\  0  \end{pmatrix}
\, e^{\lambda_1 \eta}
+
\kappa_2 
\begin{pmatrix}  0  \\  1  \end{pmatrix}
\, e^{\lambda_2 \eta}
+ \mbox{ higher order terms}
\end{equation}
as $\eta \to -\infty$.

\item At the point $Q = (1,1)$ equation \eqref{gradient} gives
$$
\nabla F \big ( 1, 1  \big )=
\begin{pmatrix}
-2  &  1  \\
-2 \tfrac{n+1}{n}  &   \tfrac{\alpha}{n\nu} c_\nu
\end{pmatrix} .
$$
The eigenvalues $\lambda_{\pm}$, computed by
\begin{equation}
\label{eigensaddle}
\lambda_{\pm}  =  \frac{ ( \tfrac{\alpha}{n \nu} c_\nu -2 ) \pm \sqrt{ \big ( \tfrac{\alpha}{n \nu} c_\nu -2 \big )^2 + 8 \tfrac{\alpha}{n\nu} } }{2} \, ,
\end{equation}
are $\lambda_- < 0$ with eigenvector $r_- = (1 , 2+\lambda_-)^T$ and  $\lambda_+ > 0$ with
eigenvector $r_+ = (1 , 2+\lambda_+)^T$.   $Q$ is a saddle point.

\item  Consider the region $\mathcal{R}$ bounded by the curves
\begin{itemize}
\item
$l_0$ is the $b$-axis defined by the equation $a=0$,
\item
$l_1$ is the horizontal line $b - 1 = 0$,
\item
$l_2$ is the parabola $b - a^2 = 0$ .
\end{itemize} 
The region $\mathcal{R}$ is expressed as 
$$
\mathcal{R} = \big \{  (a, b) :  a^2 \le b \le 1 \, , \; 0 \le a 
\le 1 \big \} ,
$$
see Figure \ref{fig:phase2}.
\begin{figure}
\centering
\includegraphics[height=3.8cm,width=10cm]{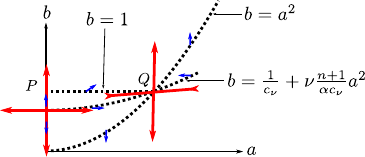}
\caption{Phase diagram of System \eqref{final}}
\label{fig:phase2}
\end{figure}
We will show that $\mathcal{R}$ is negatively invariant.

Indeed, the flow of \eqref{final} is parallel to $l_0$ as indicated in (b). On the horizontal  line $l_1 \cap \mathcal{R}$ the system \eqref{final} gives
$$
\frac{da}{d\eta} = a (1-a^2) > 0 \, , \quad 
\frac{db}{d\eta} = \tfrac{n+1}{n} (1 - a^2) > 0 \, .
$$
On the curve $l_2 \cap \mathcal{R}$ the system \eqref{final} gives
$$
\frac{da}{d\eta} = 0 \, , \quad \frac{db}{d\eta}  = 
\tfrac{\alpha}{\nu n} (a^2 -1)  < 0 \, .
$$
Therefore $\mathcal{R}$ is negatively invariant.

\item For completeness we also show that the vector $r_-$ associated to 
the eigenvalue $\lambda_-$ at $Q$ points 
towards the interior of the region $\mathcal{R}$. Recall that $Q$ is the upper right corner of the region $\mathcal{R}$
and that
\begin{itemize}
\item[] the tangent vector to $l_1$ is $\tau_1 = (1, 0)^T$ ,
\item[] the tangent vector to $l_2$ at $Q = (1,1)$ is $\tau_2 = (1,2)^T$ ,
\item[] the eigenvector $r_- = (1 , 2 + \lambda_-)^T$.
\end{itemize}
We will show that
\begin{equation}
\label{place}
0 < 2 + \lambda_- < 2 \, , \quad \forall \nu > 0 \, .
\end{equation}
This then implies that the vector $-( r_-)$ when placed at the corner $Q$ points strictly towards the interior of $\mathcal{R}$.
To show \eqref{place} recall that $\lambda_- < 0$ and use \eqref{eigensaddle} and \eqref{defce} to rewrite
$$
2 ( \lambda_- + 2) = ( \tfrac{\alpha}{n \nu} c_\nu  +  2 )  - \sqrt{ I},
$$
where
$$
\begin{aligned}
I &:=   ( \tfrac{\alpha}{n \nu} c_\nu -2 )^2 + 8 \tfrac{\alpha}{n\nu} 
\\
&=  ( \tfrac{\alpha}{n \nu} c_\nu  )^2 + 4 + 4 \tfrac{\alpha}{n \nu} -4 \tfrac{n+1}{n}
\\
&<  ( \tfrac{\alpha}{n \nu} c_\nu  )^2 + 4 + 4 \tfrac{\alpha}{n \nu}
\\
&< ( \tfrac{\alpha}{n \nu} c_\nu  + 2  )^2  .
\end{aligned}
$$
This shows  $\lambda_- + 2 > 0$ and concludes \eqref{place}.

\item It is easy to check that there can be no closed limit cycle in 
$\mathcal{R}$ because in that region
$$
\frac{da}{d\eta} = \frac{a}{b} (b - a^2) > 0 \, .
$$
The Poincar\'e-Bendixson theorem implies that there exists a heteroclinic orbit joining $( 0, \tfrac{1}{c_\nu})$ and $(1,1)$
which emanates along the direction of the stable manifold at $(1,1)$ and progresses backwards in time towards the node
at $( 0, \tfrac{1}{c_\nu})$.  \qed
\end{enumerate}

\medskip
{\it Step 4.} We turn now to defining the solution of \eqref{msys}, \eqref{mic}.
 Let $( A(\eta), B(\eta) )$ stand
for a  parametrization of the heteroclinic orbit. In view of 
\eqref{asymptotic} and the fact that $\lambda_2> \lambda_1 = 1$, there exists a 
constant $\kappa_1$ such that
\begin{equation}
\label{limit1}
e^{ - \eta}
\left (  
\begin{pmatrix}
A (\eta)
\\
B(\eta)
\end{pmatrix}
-
\begin{pmatrix}
0
\\
\tfrac{1}{c_\nu}
\end{pmatrix}
\right )
\to 
\kappa_1 
\begin{pmatrix}  1  \\  0  \end{pmatrix} ,
\qquad \mbox{ as $\eta \to -\infty$} \, ,
\end{equation}
Due to the fact that the vertical axis is a trajectory of its own and thus 
distinct from the trajectory we are considering, and due to the form of 
the phase portrait we can assert that $\kappa_1 > 0$.

For any  $\eta_0 \in \R$ the function $(A(\eta + \eta_0) , B(\eta + 
\eta_0) )$ offers an alternative
parametrization of the same heteroclinic orbit. Applying  \eqref{limit1} to this parametrization yields
\begin{equation}
\label{limit2}
e^{ - \eta}
\left ( 
\begin{pmatrix}
A (\eta + \eta_0 )
\\
B(\eta + \eta_0 )
\end{pmatrix}
-
\begin{pmatrix}
0
\\
\tfrac{1}{c_\nu}
\end{pmatrix}
\right )
\to 
\kappa_1  e^{\eta_0}
\begin{pmatrix}  1  \\  0  \end{pmatrix} ,
\qquad \mbox{ as $\eta \to -\infty$} \, .
\end{equation}
Choose $\eta_0$ such that $\kappa_1 \eta_0 = \frac{1}{\Sigma_0}$ and let 
$(a(\eta), b(\eta))$ denote the corresponding  parametrization of 
the heteroclinic :
$$
\begin{pmatrix}
a(\eta)
\\
b (\eta)
\end{pmatrix}
: =
\begin{pmatrix}
A (\eta + \eta_0 )
\\
B(\eta + \eta_0 )
\end{pmatrix} .
$$
The selected parametrization satisfies as $\eta \to -\infty$
\begin{equation}
\label{limit3}
\begin{aligned}
a(\eta) &= \frac{1}{\Sigma_0} e^\eta + o(1) e^\eta ,
\\
b (\eta) &= \frac{1}{c_\nu} + o(1) e^\eta ,
\\
\sigma (\eta) &= \frac{1}{a(\eta)} = \frac{\Sigma_0 e^{-\eta}}{1 + o(1)} = 
\Sigma_0 e^{-\eta} ( 1 + o(1) ) ,
\end{aligned}
\end{equation}
where as usual $o(1) \to 0$ as $\eta \to -\infty$.
Recalling the changes of variables \eqref{indep} and \eqref{changed} we define
\begin{alignat}{2}
U (\xi) &=   \frac{1}{\xi} u ( \log \xi) &&=   \frac{1}{\xi} \, \Big 
( \frac{a}{b} \Big ) (\log \xi) ,
\nonumber
\\
\label{defin}
\Sigma (\xi) &=  \xi \sigma (\log \xi)  &&=  \xi \, \Big ( \frac{1}{a} 
\Big ) (\log \xi) ,
\\
\Theta (\xi) &= - \tfrac{n+1}{\alpha} \log \xi  + \theta (\log \xi) &&
= - \tfrac{n+1}{\alpha} \log \xi + \tfrac{n+1}{\alpha} \big ( \log a \big ) (\log \xi ) 
- \tfrac{n}{\alpha} \big ( \log b \big )(\log \xi) \, .
\nonumber
\end{alignat}
The function \eqref{defin}  clearly satisfies the system \eqref{msys} and,
since $( a(\eta) , b(\eta) ) \to (1,1)$ as $\eta \to \infty$, 
it satisfies 
$$
\begin{aligned}
U(\xi)  &= \frac{1}{\xi} (1  + o (1)  ),
\\
\Sigma (\xi) &= \xi (1 + o (1) ),
\\[4pt]
\Theta (\xi) &=   - \tfrac{n+1}{\alpha} \log \xi + o(1),
\end{aligned}
$$
as $\xi$ goes to infinity, which shows \eqref{behavinfi}.

\medskip
{\it Step 5.} In the last step,  we study the behavior near the (apparent) singular point at $\xi =0$ and prove that the solution
is in fact regular there and satisfies  \eqref{behavior0}, \eqref{limitderiv} and \eqref{taylor}.

Clearly \eqref{limit3} and \eqref{defin}$_2$ imply that $\Sigma (\xi) \to \Sigma_0$ as $\xi \to 0$. Let $U_0$ and $\Theta_0$
be defined by \eqref{defutheta}. Since $u(\eta) = \frac{a}{b} (\eta)$,  
\eqref{limit3}  gives
\begin{equation}
\label{limit4}
u(\eta) = \frac{a}{b} (\eta) = \frac{c_\nu}{\Sigma_0} e^\eta + o(1) 
e^\eta = U_0 e^\eta + o(1) e^\eta \, .
\end{equation}
Now \eqref{defin}, \eqref{limit3}, \eqref{limit4}, \eqref{defutheta} imply
$$
\begin{aligned}
\Theta (\xi)  &=  - \tfrac{n+1}{\alpha} \log \xi + 
\tfrac{n+1}{\alpha} \log \left ( \xi \frac{1}{\Sigma_0} (1 + o(1) ) \right 
)
 - \tfrac{n}{\alpha} \log  \left ( \frac{1}{c_\nu} + o(1) \right )
 \\
 &\to \tfrac{n+1}{\alpha} \log \left (\frac{1}{\Sigma_0}  \right ) 
 - \tfrac{n}{\alpha} \log  \left ( \frac{1}{c_\nu}  \right )  = \Theta_0  
\qquad \mbox{as $\xi \to 0$}.
 \end{aligned}
$$
and together with \eqref{limit4} and \eqref{defin} provide \eqref{behavior0}. 
Observe now that
$$
\frac{dU}{d\xi} (\xi) = \frac{1}{\xi^2} \left( \frac{du}{d\eta} 
- u \right) (\log \xi) \, .
$$
Using \eqref{final}, \eqref{limit3} and \eqref{limit4}, we obtain
$$
\begin{aligned}
\frac{du}{d\eta} - u &= \frac{1}{b} \frac{da}{d\eta} - 
\frac{a}{b^2} \frac{db}{d\eta} - \frac{a}{b}
\\
&= \frac{1}{b} a (1 -\frac{a^2}{b} ) -   \frac{a}{b}  - 
\frac{a}{b^2}  \Big [ \tfrac{\alpha}{\nu n} (c_\nu b -1) - 
\tfrac{n+1}{n} a^2 \Big ]
\\
&= \tfrac{1}{n} u^3 b - u \tfrac{\alpha}{\nu n} 
\frac{c_\nu}{b} \big (b - \frac{1}{c_\nu} \big) 
\\
&= \tfrac{1}{n} \tfrac{U_0^3}{c_\nu} e^{3 \eta}(1 + o(1) )- \tfrac{\alpha}{\nu n}  U_0 o(1) e^{2\eta}  \qquad \mbox{as $\eta \to -\infty$}
\end{aligned}
$$
and thus
\begin{equation}
\label{derivbound}
\Big | \frac{dU}{d\xi} (\xi) \Big |  \le o(1) \qquad \mbox{as $\xi \to 0$}.
\end{equation}
Since $U \in C[0,1] \cap C^1(0,1)$ the mean value theorem implies
$$
\left |  \frac{U(\xi) - U_0}{\xi}  \right | = | U' (\xi^*) | \quad \mbox{for some $\xi^* \in (0, \xi)$},
$$
and \eqref{derivbound} implies  $U'(0)$ exists and $U'(0) = 0$. We conclude that $U \in C^1 [0,\infty)$ has the Taylor expansion \eqref{taylor}$_1$.
In turn,  $\Sigma \in C^2 [0, \infty)$,  $\Sigma' (0) = 0$, $\Sigma''(0) = \tfrac{1}{2} U_0$ and 
$\Sigma$ achieves the Taylor expansion \eqref{taylor}$_2$.

Turning to $\Theta$, observe that by \eqref{limit3} and \eqref{defin}$_3$
$$
\begin{aligned}
\theta (\eta) &= \tfrac{n+1}{\alpha} \log a(\eta) - 
\tfrac{n}{\alpha} \log b (\eta)
\\
&= \tfrac{n+1}{\alpha}  \log \left ( \frac{1}{\Sigma_0} e^\eta (1 + o(1) )
\right )  - \tfrac{n}{\alpha} \log \left ( \frac{1}{c_\nu} (1 + o(1) e^\eta )  
 \right )
\\
&= \tfrac{n+1}{\alpha} \log \frac{1}{\Sigma_0} - \tfrac{n}{\alpha} \log \frac{1}{c_\nu} + \tfrac{n+1}{\alpha}  \log e^\eta  + \tfrac{n+1}{\alpha}  \log ( 1 + o(1)) 
- \tfrac{n}{\alpha} \log \big ( 1 + o(1) e^\eta \big ) 
\\
&= \Theta_0 + \tfrac{n+1}{\alpha} \eta + \tfrac{n+1}{\alpha}  \log ( 1 + o(1))  - \tfrac{n}{\alpha} \log ( 1 + o(1) ) 
\, , 
\qquad \mbox{as $\eta \to -\infty$} \, ,
\end{aligned}
$$
while
$$
\begin{aligned}
\Theta (\xi) &= - \tfrac{n+1}{\alpha} \log \xi + \theta(\log \xi)
\\
&= \Theta_0 + o(1)  \qquad \mbox{as $\xi \to 0$}.
\end{aligned}
$$
In addition,
$$
\Theta'(\xi) = - \tfrac{n+1}{\alpha} \frac{1}{\xi} + \frac{1}{\xi}  \frac{d\theta}{d\eta} ( \log \xi),
$$
combined with the formula
$$
\begin{aligned}
\frac{d\theta}{d\eta} &= \tfrac{n+1}{\alpha} \frac{1}{a} \frac{da}{d\eta} 
- \tfrac{n}{\alpha} \frac{1}{b} \frac{db}{d\eta}
\\
&= \tfrac{n+1}{\alpha} \left( 1 - \frac{a^2}{b} \right ) - \tfrac{1}{\nu} 
\left ( c_\nu - \frac{1}{b} \right ) + \tfrac{n+1}{\alpha} 
\frac{a^2}{b}
\\
&= \tfrac{n+1}{\alpha} - \tfrac{1}{\nu} \frac{c_\nu}{b} \left ( b - 
\tfrac{1}{c_\nu} \right ),
\end{aligned}
$$
gives
$$
\begin{aligned}
\Theta' (\xi) &= - \frac{1}{\xi} \,  \left (  \frac{c_\nu}{b } \big( 
b - \tfrac{1}{c_\nu} \big ) \right ) (\log \xi)
\\
&= - \frac{1}{\xi}  \,  \frac{c_\nu}{ b } \big(\log \xi \big) \,  o(1) \xi
\\
&= o(1) \to 0 \qquad  \qquad \mbox{as $\xi \to 0$} \, .
\end{aligned}
$$
This shows $\Theta \in C^1 [0, \infty)$ has the Taylor expansion \eqref{taylor}$_3$ at the origin.

\end{proof}

\section{Localization - Part II}
\label{seclocalII}

We are ready to state and prove the main result of the paper. This  concerns the adiabatic case $\kappa =0$
of system \eqref{shear}, taken now on the whole real line 
$(x,t) \in \R \times \R^+$ and expressed in
terms of the variable $u=v_x$ in the form:
\begin{equation}
  \label{adiashear} 
  \begin{aligned}
    & u_{t} = \sigma_{x x} ,
    \\
    & \theta_{t} =  \sigma u ,
    \\
    & \sigma  = e^{-\alpha \theta} u^n \, .
  \end{aligned}
\end{equation}
Our analysis is based on the transformations of section \ref{seclocal} and 
the self-similar 
solution constructed in section \ref{secss}.

\begin{theorem}
\label{shearband} 
Let $\theta_0 > 0$ and $\Sigma_0 > 0$ be given. For any $\lambda > 0$, 
the system \eqref{adiashear} admits a special class of solutions describing localization of the form:
\begin{align}
\label{locu}
u(x,t) &= \left ( \tfrac{\alpha}{c_0} t + 1 \right )^{\tfrac{\lambda}{\alpha}} U_\lambda \left ( \sqrt{\lambda}  x  \left ( \tfrac{\alpha}{c_0}  t + 1 \right )^{\tfrac{\lambda}{\alpha}}   \right ), 
\\[2pt]
\label{locsigma}
\sigma (x,t) &= \sigma_s (t)  \left ( \tfrac{\alpha}{c_0}  t + 1 \right )^{ - \tfrac{\lambda}{\alpha}}
\Sigma_\lambda \left ( \sqrt{\lambda}  x  \left ( \tfrac{\alpha}{c_0} t + 1 \right )^{\tfrac{\lambda}{\alpha}}   \right ) ,
\\[2pt]
\label{loctheta}
\theta(x,t) &= \left ( 1 +\lambda \tfrac{n+1}{\alpha} \right ) \theta_s (t) - \lambda \tfrac{n+1}{\alpha} \theta_0 +
\Theta_\lambda \left ( \sqrt{\lambda}  x  \left ( \tfrac{\alpha}{c_0} t + 1 \right )^{\tfrac{\lambda}{\alpha}}   \right ) ,
\end{align}
where $\sigma_s$, $\theta_s$ are the uniform shear solutions in \eqref{expuss}, $c_0 = e^{\alpha \theta_0}$, while $(U_\lambda, \Sigma_\lambda, \Theta_\lambda)(\xi)$
is an even function defined on $\R$, which solves the initial value problem \eqref{msys}, \eqref{mic} with $\nu = \lambda > 0$ for $\xi \in [0,\infty)$.
\end{theorem}

\begin{remark} \rm
\label{rmkbeh}
\begin{enumerate}
\renewcommand{\theenumi}{\roman{enumi}}
\item[]
\item  Because $\left ( \tfrac{\alpha}{c_0}  t + 1 \right 
)^{\tfrac{\lambda}{\alpha}}$ goes to infinity in time, the level curves of 
$$
\xi (x,t) = \sqrt{\lambda}  x  \left ( \tfrac{\alpha}{c_0}  t + 1 \right )^{\tfrac{\lambda}{\alpha}},
$$ 
accumulate along the $\lbrace x=0\rbrace$ axis in the $(x,t)$ plane. The functions \eqref{locu}, \eqref{locsigma}, \eqref{loctheta} thus describe a localizing
solution  where the flow tends to concentrate around $x=0$ as the time proceeds.

\item  The solution \eqref{locu}, \eqref{locsigma}, \eqref{loctheta} is 
defined on $\R \times \R^+$ and emanates from initial data
\begin{equation}
\label{initialdata}
\begin{aligned}
u(x,0) &=  U_\lambda \big ( \sqrt{\lambda}  x    \big  ) ,
\\[2pt]
\sigma (x,0) &= \sigma_s (0)  
\Sigma_\lambda  \big ( \sqrt{\lambda}  x    \big ) ,
\\[2pt]
\theta(x,0) &=  \theta_s (0)  +
\Theta_\lambda  \big  ( \sqrt{\lambda}  x     \big )  \, .
\end{aligned}
\end{equation}
The initial data depend on two parameters: 
The parameter $\lambda$ which can be thought as a length scale in initial data, and $\Sigma_0$ in \eqref{mic} which describes the amplitude of the initial perturbation. 

\item  This parametric family of solutions is valid for any $\lambda > 0$. 
If $\lambda$ is large then 
the length scale of the initial perturbation is short while  the growth in 
time and the rate of localization for the solution become fast.
\end{enumerate}
\end{remark}

\begin{proof}
We introduce the changes of variables \eqref{resctran} and  the stationary variant of \eqref{adavar}, which once combined together
reads
\begin{equation}
\label{transform}
\begin{aligned}
u (x, t ) &= \frac{1}{r(\tau(t) )} \bU \Big ( \frac{x}{r(\tau(t) ) }  \Big ) ,
\\
\sigma (x, t)  &= \sigma_s (t)  r(\tau(t) ) \, \bS \Big ( \frac{x}{r(\tau(t) ) }  \Big ),
\\
\theta (x, t)  &=  \theta_s (t) - \frac{n+1}{\alpha} \log r(\tau(t) )  + \bTheta \Big ( \frac{x}{r(\tau(t) ) } \Big )  \, .
\end{aligned}
\end{equation}
We select $r(\tau) = e^{-\lambda \tau}$ with $\lambda > 0$. 
According to the analysis in section \ref{secinvariance} the function $(\bU, \bS, \bTheta)(\xi)$ is sought as a solution of
\eqref{sysxi} defined on the interval $(-\infty, \infty)$. 

The desired $(\bU, \bS, \bTheta)(\xi)$ is constructed as follows. First, observe the relation between
the system \eqref{sysxi}  and the system \eqref{msys}.  Let $\Sigma_0$ be fixed, and let 
$(U_\lambda , \Sigma_\lambda, \Theta_\lambda)(\xi)$ be the solution of \eqref{msys} with $\nu = \lambda$ and \eqref{mic},
defined for $\xi \in [0, \infty)$. This solution is well defined by Theorem \ref{connection}.  Using a direct computation,
the smoothness properties 
in (iii) of Theorem \ref{connection},  and \eqref{scaleinv2} we conclude that
\begin{enumerate}
\renewcommand{\theenumi}{\alph{enumi}}
\item
The function  $(\bU, \bS, \bTheta)$ defined by
\begin{equation}
\label{solscale}
\bU (\xi) =  U_\lambda (\sqrt{\lambda} \xi) \, , \quad \bS (\xi) =  \Sigma_\lambda (\sqrt{\lambda} \xi) \, , \quad
\bTheta (\xi) =  \Theta_\lambda (\sqrt{\lambda} \xi) ,
\end{equation}
satisfies \eqref{sysxi} on $\R^+$ and the condition $\bS (0) = \Sigma_0$.
\item If we extend $(\bU, \bS, \bTheta)$  on $(-\infty, 0)$ by setting  
 $(\bU, \bS , \bTheta)(-\xi) = ( \bU, \bS, \bTheta) (\xi)$ for $\xi < 0$, 
the resulting function $\bU, \bTheta \in C^1 (\R)$, $\bS \in C^2 (\R)$ is a 
solution of \eqref{sysxi} defined on the whole real line that is even.
\end{enumerate}
Noting that 
$$
r(\tau (t)) = e^{- \lambda \tau(t)} = \left ( \tfrac{\alpha}{c_0} t + 1 \right )^{ - \tfrac{\lambda}{\alpha}} ,
$$
we combine \eqref{transform} with \eqref{solscale} and \eqref{expuss} to arrive at \eqref{locu}, \eqref{locsigma}, \eqref{loctheta}. 

\end{proof}

In Figure \ref{uqsol} the profiles  of $u(x,t)$ and $\theta(x,t)$ in \eqref{locu} and \eqref{loctheta}, respectively, are drawn at various instances of time from $t=0$ to $t=200$.
To achieve these profiles, we construct numerically the heteroclinic orbit for  \eqref{final} whose existence is justified in Proposition \ref{thode}. 
The construction is simple, if one exploits the special properties of system \eqref{final}, 
and proceeds as follows: Initial conditions $(a_0,\ b_0)$ are selected close to the equilibrium $Q = (1,1)$  inside the region $\mathcal{R}$ 
and in the direction of the stable manifold of $Q$ for  the linearized problem. Then  \eqref{final} is solved backward in time, using an ode solver, and since $P = (0, 1/c_\nu) $ is
an unstable node (for the forward problem)  the resulting orbit is a good approximation of the heteroclinic. 
For the  integration of  \eqref{final} we use an explicit Runge--Kutta predictor-corrector method of order (4,5) with adaptive time stepping. 
Then  $\Sigma_0$ is selected by choosing  an appropriate parametrisation $\eta_0$ following the construction in \eqref{limit2}.
\begin{figure}[h]
\centering
\includegraphics[height=6cm,width=8cm]{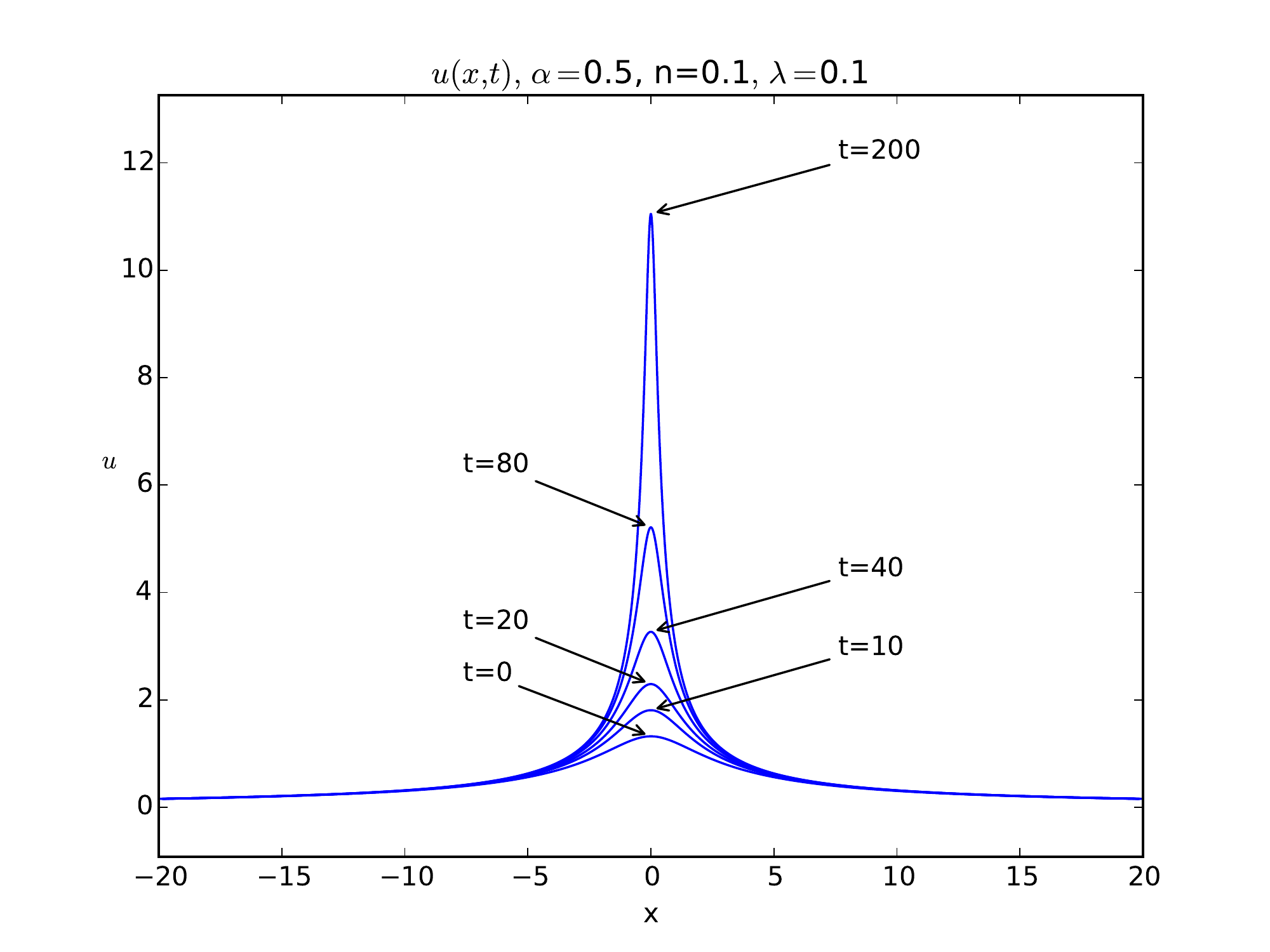}
\includegraphics[height=6cm,width=8cm]{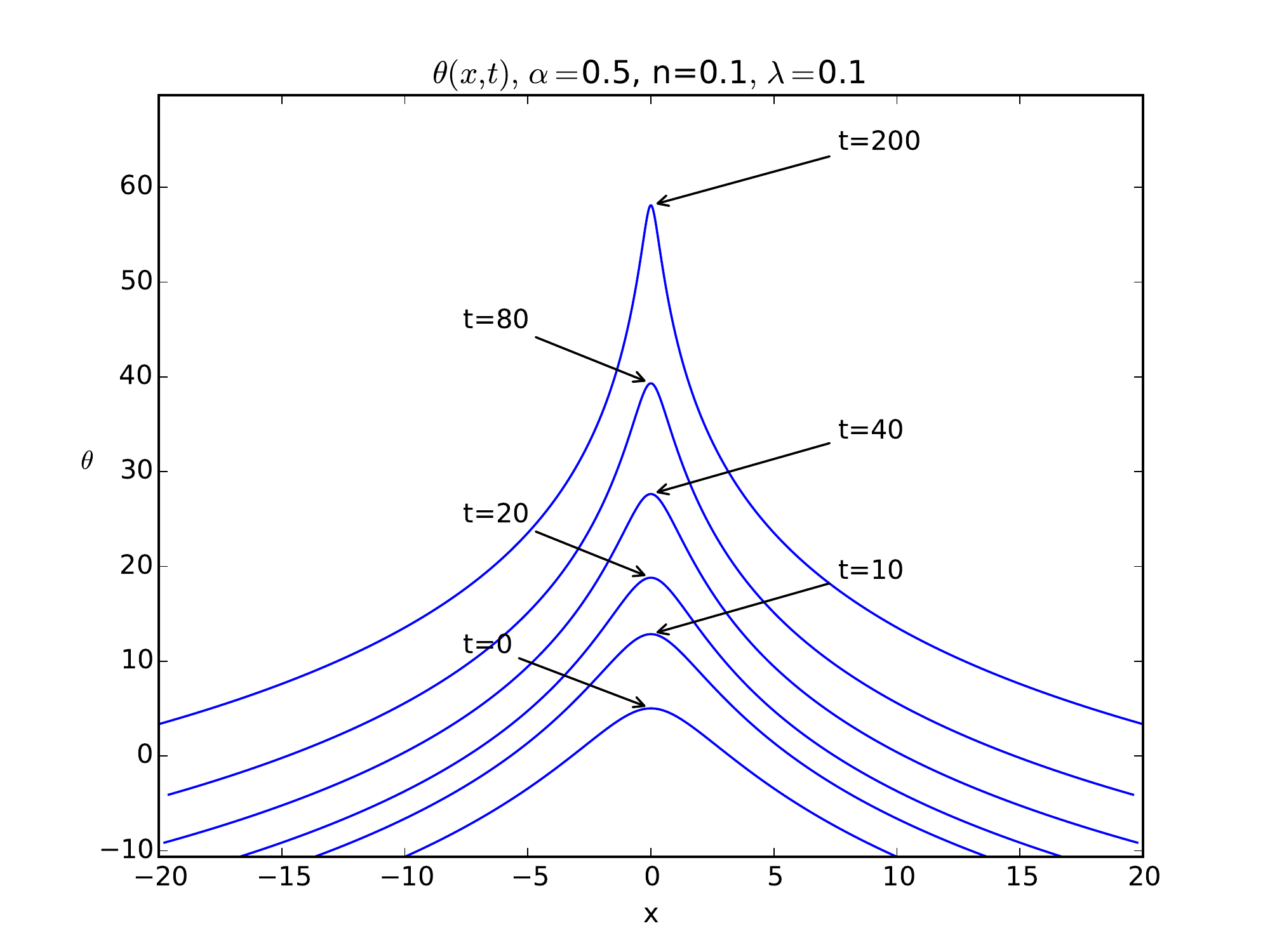}
\caption{The solution $u(x,t)$, \eqref{locu} and  $\theta(x,t)$, \eqref{loctheta}.}
\label{uqsol}
\end{figure}

For the numerical runs the constitutive parameters are $n = 0.1$, $\alpha = 0.5$. The parameters related to the solution are $\theta_0=10$,  $\lambda = 0.1$ while for the
selected reparametrization of the heteroclinic orbit the value of $\Sigma_0=1.88$.  Once the heteroclinic has been constructed, 
the profiles $U_{\lambda}, \ \Theta_{\lambda}$ are computed via the changes of variables \eqref{changed} and \eqref{indep}. Finally,  $u(x,t)$ and $\theta(x,t)$ are computed by \eqref{locu}, \eqref{loctheta} respectively with $\theta_0=10$, and are presented in Fig \ref{uqsol}. 


\section{Summary - Conclusions}
\label{summary}

We considered the system \eqref{shear} as a paradigm to study the onset of localization and formation of shear bands
and to examine the role of thermal softening, strain-rate hardening and thermal diffusion in their development.
The stability of the uniform shear solution \eqref{expuss} leads, in general,  to the analysis of a non-autonomous system. 
The present  model has the remarkable property that the system of relative perturbations \eqref{rescaled} is autonomous 
in the adiabatic case ($\kappa=0$).
A study of linearized stability for the case with no thermal diffusion  $\kappa =0$ shows that: 
(i) For $n = 0$ the high frequency modes grow exponentially fast,  with the exponent proportional to the frequency,
what is characteristic of ill-posedness and  Hadamard instability. 
(ii) When $n > 0$ the modes are still unstable but the rate of growth is bounded independently of the frequency.
The behavior in that range is that characteristic of Turing instability.
The effect of thermal diffusion ($\kappa > 0$) for the 
non-autonomous linearized model \eqref{introlinrelexp} was  also assessed; it was shown that perturbations
grow in the early stages of deformation,  but over time  the compound effect of diffusion
stabilizes the response of the system.

Then we turned to the nonlinear system \eqref{adia} and constructed a class of self-similar solutions on the real line, leading to the explicit
solutions \eqref{locu}, \eqref{locsigma}, \eqref{loctheta}. Their construction is based on  the self-similar profiles of \eqref{msys}, \eqref{mic}, 
in turn determined by a suitable heteroclinic connection. They depend on three parameters: $\theta_0$ linked to the  uniform shear in \eqref{ARUSS}, 
a parameter $\lambda$ which can be viewed  as a length scale of initial data, and  $\Sigma_0$ determining the size of the initial profile. 
The solutions presented in Fig \ref{uqsol}  evidently exhibit localisation:  as time increases a coherent structure forms with the deformation
localizing in a narrow band. In this self-similar solution there is no trace of the oscillations.
A comparison with the modes of the linearized problem indicates that the nonlinearity suppresses
the oscillations and a seamless coherent structure emerges as the net outcome of Hadamard instability, rate dependence, and the
nonlinear structure of the problem.


\end{document}